\documentclass[reqno,10pt,centertags]{amsart} 
\usepackage{amsmath,amsthm,amscd,amssymb,latexsym,esint,upref,stmaryrd,enumerate,xcolor,verbatim,yfonts,multicol}
\calclayout
\usepackage{color,float}
\usepackage{hyperref}

\usepackage{tikz}
\usetikzlibrary{calc,decorations.markings}
\usepackage{pgfplots}
\pgfplotsset{compat=1.15}

\usepackage{graphicx}
\usepackage{enumitem}

\newcommand*{\mailto}[1]{\href{mailto:#1}{\nolinkurl{#1}}}
\newcommand{\arxiv}[1]{\href{http://arxiv.org/abs/#1}{arXiv:#1}}

\newcommand{\R}{{\mathbb R}}
\newcommand{\N}{{\mathbb N}}
\newcommand{\Z}{{\mathbb Z}}
\newcommand{\C}{{\mathbb C}}

\newcommand{\bbC}{{\mathbb{C}}}

\newcommand{\bbN}{{\mathbb{N}}}

\newcommand{\bbR}{{\mathbb{R}}}

\newcommand{\bbZ}{{\mathbb{Z}}}

\newcommand{\cB}{{\mathcal B}}

\newcommand{\cH}{{\mathcal H}}

\newcommand{\cK}{{\mathcal K}}

\newcommand{\beq}{\begin{align}}
\newcommand{\enq}{\end{align}}

\renewcommand{\a}{\alpha}

\renewcommand{\r}{\rho}
\newcommand{\x}{\xi}
\newcommand{\z}{\zeta}

\DeclareMathOperator{\Ai}{Ai}

\newcommand{\Arg}{\text{\rm Arg}}
\newcommand{\Res}{\text{\rm Res}}
\renewcommand{\Re}{\text{\rm Re}}
\renewcommand{\Im}{\text{\rm Im}}
\renewcommand{\ln}{\text{\rm ln}}

\newcommand{\Norm}[1]{\left\lVert#1\right\rVert}

\newcommand{\lb}{\label}

\newcommand{\dott}{\,\cdot\,}

\newcommand{\bi}{\bibitem}

\let\geq\geqslant
\let\leq\leqslant

\makeatletter
\def\theequation{\@arabic\c@equation}

\allowdisplaybreaks 
\numberwithin{equation}{section}

\newtheorem{theorem}{Theorem}[section]

\newtheorem{lemma}[theorem]{Lemma}
\newtheorem{corollary}[theorem]{Corollary}
\newtheorem{definition}[theorem]{Definition}

\newtheorem{assumption}[theorem]{Assumption}
\theoremstyle{remark}
\newenvironment{remark}[1][]{\refstepcounter{theorem}\par\medskip\noindent\textit{Remark~$\theexample. #1$} \rmfamily}{{\ }\hfill $\diamond$ \vspace{6pt}}

\begin{document}

\title{$\zeta$-functions via contour integrals and universal sum rules} 

\author[G.\ Fucci]{Guglielmo Fucci}
\address{Department of Mathematics, 
East Carolina University, 331 Austin Building, East Fifth St.,
Greenville, NC 27858-4353, USA}
\email{\mailto{fuccig@ecu.edu}}
\urladdr{\url{http://myweb.ecu.edu/fuccig/}}

\author[M. Piorkowski]{Mateusz Piorkowski}
\address{Department of Mathematics, KTH Royal Institute of Technology, Stockholm, Sweden}
\email{\mailto{mateuszp@kth.se}}
\urladdr{\url{https://sites.google.com/view/mateuszpiorkowski/home?pli=1}}

\author[J.\ Stanfill]{Jonathan Stanfill}
\address{Division of Geodetic Science, School of Earth Sciences, The Ohio State University \\
275 Mendenhall Laboratory, 125 South Oval Mall, Columbus, OH 43210, USA}
\email{\mailto{stanfill.13@osu.edu}}
\urladdr{\url{https://u.osu.edu/stanfill-13/}}


\date{\today}

\begin{abstract}
This work develops an analytic framework for the study of the $\zeta$-function associated with general sequences of complex numbers. We show that a contour integral representation, commonly used when studying spectral $\zeta$-functions associated with self-adjoint differential operators, can be extended far beyond its traditional setting. In contrast to representations utilizing integrals of $\theta$-functions, our method applies to arbitrary sequences of complex numbers with minimal assumptions. This leads to a set of universal identities, including sum rules and meromorphic properties, that hold across a broad class of $\zeta$-functions. Additionally, we discuss the connection to regularized (modified) Fredholm determinants of $p$-Schatten--von Neumann class operators. We illustrate the versatility of this representation by computing special values and residues of the $\zeta$-function for a variety of sequences of complex numbers, in particular, the zeros of Airy functions, parabolic cylinder functions, and confluent hypergeometric functions. Furthermore, we employ the adaptive Antoulas--Anderson (AAA) algorithm for rational interpolation in the study of the Airy $\zeta$-function.
\end{abstract}

\maketitle



\section{Introduction} \lb{s1}

Studying the properties of $\zeta$-functions associated with sequences of complex numbers has proven immensely important in many areas of mathematics and physics, see for example \cite{Bar92,Be86,BVW88,BKMM09,BCEMZ03,Cog06,CZ97,CZ04,DK78,Em12,Em94,FLP,Ki02,Ki08,RS71,Sp06,voros23,Za94}.
The aim of this work is to illustrate that the contour integral representation is a very useful and versatile tool for analyzing the properties of the $\zeta$-function since it can be utilized even when other representations are not valid or do not readily provide the desired information (for an instance in which this occurs see Remark \ref{sub1.1}). This makes the contour integral representation arguably the most far-reaching of any such representations and, in the course of this work, we will show that it can be utilized to quickly and efficiently calculate not only special values of the $\z$-function but also residues at its poles. 

The contour integral representation has been utilized numerous times in the literature to study the $\z$-function associated with sequences of real numbers that are bounded from below.  
For instance, it has been used to great effect in the analysis of the spectral $\zeta$-functions associated with certain  self-adjoint operators bounded from below \cite{FGKS21,FPS25,FS25,GK19,Ki02,KM03}. 
Since the dynamics of the vast majority of physical systems is described by self-adjoint differential operators, almost all of the available literature on $\z$-functions and their analytic continuation is limited to the case of sequences of real numbers that are bounded from below. In particular, the process of analytic continuation has been described several times in this setting and consists of two main steps: first, the contour is suitably deformed in order to obtain a representation valid in a strip of the complex plane, and, second, the ensuing representation is analytically continued to the left by subtracting and then adding the asymptotic expansion of the integrand (see Section \ref{continuation} for the details of this procedure).

The main result of this paper is a proof that the integral representation of the $\z$-function and the method of analytic continuation that has been developed for the case of real sequences bounded from below can be extended, with few modifications, to the general case of sequences of complex numbers satisfying only minimal conditions (see Definition \ref{defzeta}). This integral representation immediately implies universal sum rules that surprisingly hold for all such $\zeta$-functions (see Theorem \ref{t4}). These universal sum rules are the complete generalization of the sum rules satisfied by sequences of numbers such as the inverse squares of Bessel function zeros (see, e.g., \cite{AB96}, \cite[Sec.~7.9]{EMOT53II}, \cite{Ra74,Sn60}) and Airy function zeros \cite{voros}.  This illustrates the amazing result that such sum rules are, in fact, ubiquitous; we write them out explicitly in the Airy case studied in Sec.~\ref{Sect:Airy_Zeros}. We further discuss in Remark \ref{r2.10} $(iii)$ so-called \emph{exact} sum rules studied in the mathematical physics literature \cite{Ka25,voros23}. 

Another interesting phenomenon that we encounter are $\zeta$-functions that can be analytically continued to be \emph{entire} functions (see Sec.~\ref{Sect:Zeros_Para}, \ref{Sect:Zeros_Conf}). This is in stark contrast to the case of sequences of real numbers bounded from below, for which the $\zeta$-function always develops at least one real singularity on the boundary of the region of convergence (e.g., the simple pole of the Riemann $\zeta$-function at $s = 1$). 

To the best of our knowledge, the systematic extension of the contour integral methods previously employed to study the $\z$-function of real sequences bounded from below  to more general sequences of complex numbers is presented here for the first time. As a result, these methods can be utilized for $\zeta$-functions associated with the spectrum of more general operators appearing in physical applications, including those unbounded above and below and those that are non-self-adjoint such as sectorial operators. Interesting examples include Dirac operators, Jacobi operators, Sturm--Liouville operators with sign-changing or complex coefficients, and higher order differential operators. It is very important to mention that the general methods developed in this work can be applied to the field of perturbations of classical gravitational backgrounds involving black holes or branes \cite{bert09}. The analysis of these phenomena naturally leads to quasinormal modes which are linked to the energy of a dissipative system. 
Quasinormal modes, in the ambit of black holes, arise naturally when trying to solve the non-hermitian eigenvalue problem of a differential operator describing the linear perturbations of a fixed gravitational background \cite{bert09}. In the context of semiclassical quantization of black holes, functional determinants play an important role and they have been described in terms of products of suitable quasinormal modes (see, e.g., \cite{dene10}). The results obtained in this work, especially those presented in Section \ref{continuation}, can provide perhaps a simpler and more direct way of computing the functional determinant of quasinormal modes by using the derivative of the corresponding $\z$-function at $s=0$. Until now this approach would not have been possible, but our work shows that the ordinary analytic continuation techniques can be extended to $\z$-functions of such complex sequences.   

It is important to remind the reader that the study of the analytic continuation of the $\zeta$-function is of utmost importance since many values of interest in physical applications (i.e., spectral $\zeta$-functions) lie to the left of the convergence region of the sum definition. Such values of interest appear in zeta regularization techniques \cite{Em94,FP17,Ha77}, the $\zeta$-regularized functional determinant $\exp(-\zeta'(0))$ \cite{FGKS21,FPS25,FS25,GK19,Har22,KM03,KM04}, heat kernel coefficients as residues of poles \cite{BKD96,Cog06,Ki06}, and the Casimir energy via the value at $-1/2$ \cite{BVW88,BKMM09,Dwk84}. Our work now allows to analyze these values of the $\z$-functions associated with a larger number of physical systems not previously treatable (such as dissipative systems).  

In addition to these values, the so-called special values of the $\zeta$-function (i.e., at integer values to the left of the original region of convergence) are of general mathematical interest. Our methods not only allow one to efficiently detect the presence of singularities in the analytic continuation, but also yield precise formulas for their residues in case of poles and their values in case of regular points. Moreover, the methods immediately extend beyond these assumptions allowing general cases to be addressed similarly (see, e.g., Remark \ref{rem3.2} and \cite{FS25}).

As a final comment, we would like to mention that the integral representation and its method of analytic continuation outlined here allowed us in Section \ref{Sect:Airy_Zeros} to consider, as an example, the Airy $\z$-function. This function has been studied before in the literature \cite{Cr96,voros23} but only its values at the positive integers $n \geq 2$ were known. With very little effort we were able not only to recover those previously known values, but also to compute the values at all integer $n \leq 1$, including the value of $\zeta_{\Ai}(1)$ confirming 
a conjecture due to Crandall \cite{Cr96}. It turns out that at all negative integers not multiples of 3 the Airy $\z$-function vanishes, showing that they are its trivial zeros, a previously unknown fact. We also illustrate how the adaptive Antoulas--Anderson (AAA) algorithm for rational interpolation introduced in \cite{AAA} can be employed in the study of the Airy $\zeta$-function.   

\subsection{Organization of paper} Section \ref{s2} begins with Definition \ref{defzeta}, which introduces the general sequences of complex numbers to which we will associate a $\zeta$-function. We then introduce an entire \textit{characteristic function} (see \eqref{c1} and Remark \ref{rem2.3}) which encodes the sequence as its zeros. From there, we will prove that this function satisfies the necessary properties that allow us to define a contour integral representation of the $\zeta$-function, see \eqref{2.17}. As an immediate application, we illustrate a general deformation process while discussing its implications such as  universal sum rules for admissible positive integer values, then connect this representation to regularized Fredholm determinants.

The remainder of the paper illustrates how these results enable one to analyze the structure of the analytic continuation of the $\zeta$-function. In particular, in Sections \ref{continuation} and \ref{s4}, we find exact formulas for special values and residues of the $\zeta$-function, as well as the associated $\zeta$-regularized determinant, through the asymptotic behavior of the characteristic function for large argument. These general results are then applied to multiple examples of interest in Section \ref{s5}, starting with the classic Riemann and Hurwitz $\zeta$-functions, before transitioning to $\zeta$-functions associated with zeros of Airy functions, parabolic cylinder functions, and confluent hypergeometric functions. Note that such zeros often appear in applications as the spectrum of certain differential operators.

\section{The \texorpdfstring{$\zeta$}{zeta}-function and its universal properties}\label{s2}

The general theory of $\z$-functions of sequences provides the following necessary conditions $(i)$, $(ii)$ below, under which \eqref{introz} is well-defined \cite{Sp06}.  We additionally impose condition $(iii)$, as it will allow us to obtain convenient integral representations for the $\z$-function.
\begin{definition}\label{defzeta}
    Given a sequence of complex numbers, $S=\{a_n\}_{n\in\N}$, such that 
    \begin{enumerate}[label=$(\roman*)$]
    \item $0<|a_1|\leq|a_2|\leq \dots \to\infty$,
    \item the exponent of convergence $\a=\limsup_{n\to\infty}(\ln \, n)/(\ln|a_n|)$ is finite,
    \item  there exists $\varepsilon>0$ and $\Psi\in[-\pi,\pi)$ such that all $a_n\in \mathcal{I}_\varepsilon$, where 
    \begin{align}\label{I_eps}
        \mathcal{I}_\varepsilon=\left\{z\in\C : \Arg(z) \not \in (\Psi-\varepsilon, \Psi+\varepsilon)\right\}.
    \end{align}
    \end{enumerate}
    Then we can define its associated $\zeta$-function as
\begin{align}\label{introz}
    \zeta_{S}(s)=\sum_{n=1}^{\infty}a_{n}^{-s},\quad \Re(s)>\alpha,
\end{align}
which converges uniformly and absolutely in its domain of definition. 
\end{definition}

Note that the $a_n^{-s}$ are entire functions in $s$, however we still need to choose a branch cut in the definition of the mapping $a \mapsto a^{-s}$ for $s \in \C \setminus \Z$. Due to $(iii)$ in Definition \ref{defzeta}, it is possible to choose the branch cut to lie on $R_\Psi=\{z=te^{i\Psi}:t\in [0,\infty)\}$. 
It is important to point out that item $(iii)$ appears as part of Definition \ref{defzeta} largely for convenience since it allows us to use the simple ray $R_{\Psi}$ as a branch cut. We could have, instead, more generally assumed that there is some curve connecting $0$ to $\infty$ defining the branch cut of the map $a \mapsto a^{-s}$. In that case, the curve must lie in $\mathbb C \setminus \mathcal B_\delta$, defined in \eqref{2.11} for some $\delta > 0$.

\begin{remark}\label{sub1.1}
We now briefly make our previous comments regarding the use of the Mellin transform to produce a representation of $\zeta$-functions more precise. A detailed description of these ideas can be found in \cite{jor93}. If a sequence of complex numbers $B=\{b_{n}\}_{n\in\N}$ satisfies the condition $(i)$ of Definition \ref{defzeta} {\it and} tends to infinity in a sector contained in the right half of the complex plane, then we can define the $\theta$-function associated with that sequence as
\begin{align}
    \theta_B(t)=\sum_{n=1}^{\infty}e^{-b_n t}, \quad t>0. 
\end{align}
This function is of particular importance since it can be related to the $\z$-function of a sequence through the Mellin transform \cite{jor93}. In fact, let us assume that the sequence $S=\{a_n\}_{n\in\N}$ satisfies the assumptions of Definition \ref{defzeta} and that $a_n$ belong to a sector properly contained in the right half of the complex plane.
Then, by using the Mellin transform of the exponential, that is,
\begin{align}
\int_0^\infty t^{s-1}e^{-\lambda t}\, dt=\Gamma(s)\lambda^{-s},\quad \Re(\lambda),\Re(s)>0,
\end{align}
we can write 
\begin{align}\label{rem2.5}
\zeta_S(s)=\sum_{n=1}^\infty a_n^{-s}=\frac{1}{\Gamma(s)}\int_0^\infty t^{s-1}\theta_S(t)\, dt,\quad \Re(s)>\alpha.
\end{align}
Let us add here that we can treat the case of (necessarily) finitely many terms in the sequence $S=\{a_n\}_{n\in\N}$ having real part negative since we can separate them from the sum and proceed with the Mellin transform for the remaining infinitely many terms with positive real part.  
According to Theorem 1.5 of \cite{jor93}, if $\theta_S(t)$ is bounded for $t\to\infty$ and has, as $t\to 0^+$, an asymptotic expansion of the form
\begin{align}\label{rem2.6}
    \theta_{S}(t)=\sum_{j = 1}^{q-1}b_{p_j}(\ln\,t)t^{p_j}+O\left(t^{p_q}|\ln\,t|^{m(p_q)}\right), \quad m(p_q)=\textrm{max deg}\,b_p,
\end{align}
with $P=(p_1,p_2,\ldots,p_q)$ a tuple of strictly increasing real numbers and $b_{p}(t)$ polynomials, then $\z_S(s)$ in \eqref{rem2.5} has an analytic continuation to a meromorphic function in the region $p_q<\Re(s)\leq \alpha$. The poles of $\z_S(s)$ are located at the points $\{-p\}$ of $P$. When $P$ is infinite and $p_q \to \infty$ as $q \to \infty$ (namely the expansion \eqref{rem2.6} is valid to all orders) then one can extend $\z_S(s)$ to a meromorphic function in $\Re(s)\leq\alpha$.     

It is clear from this description that in order to perform the analytic continuation of $\z_{S}(s)$ one needs to be assured of the existence of the asymptotic expansion of $\theta_S(t)$ as $t\to 0$ of the form given in \eqref{rem2.6}. This, in practice, is somewhat difficult to do \cite{jor93} but there are notable cases in which the $\theta$-function has been shown to have such an asymptotic expansion. For instance, the $\theta$-function associated with the spectrum of a Laplace operator, also known as the trace of the heat kernel, on Riemannian manifolds with or without a boundary has an expansion of the type \eqref{rem2.6} but without logarithmic terms. More cases obviously exist in one or higher dimensions (it is impossible to exhaustively discuss the vast amount of literature that exists on the heat kernel and its trace and corresponding small-$t$ asymptotic expansion, therefore we refer the reader to the compendium \cite{Vas03} and references therein).  

However, the requirement that $\theta_S(t)$ possesses a small-$t$ expansion of the form \eqref{rem2.6} can be too restrictive for certain situations. In fact, an asymptotic expansion of the form \eqref{rem2.6} can only allow for the development of poles (simple or of higher order when logarithmic terms are present) in the corresponding $\zeta$-function. There are, however, settings in which the $\z$-function develops branch points (see, e.g., \cite{FPS25,FS25,Ki06,Ki08}) and these problems cannot be studied by using the Mellin transform representation of the $\z$-function. 
The second problem lies in the fact that the Mellin transform representation is valid under the assumption that the terms of the sequence $S$ be in a sector properly contained in the right half of the complex plane. Although  this assumption is not an issue in physical applications and in the spectral theory of self-adjoint operators, it can be restrictive for the study of the $\z$-function of more general complex sequences (which is the goal of this paper). Lastly, the representation \eqref{rem2.5} provides poles and residues of the $\z$-function very easily once the expansion \eqref{rem2.6} is known.

In stark contrast to this, the contour integral representation we prove below is valid in all of these settings without any modification, showing how versatile it really is. This includes the setting in which the $\z$-function possesses singularities other than poles. Indeed, by allowing for a more general  asymptotic expansion of the characteristic function in Assumption \ref{assump2.9}, one can obtain an analytic continuation similar to the one outlined in the proof of Theorem \ref{t2.11} but showing the presence of branch points. This analysis has been recently done in the case of the P\"oschl--Teller potential in \cite{FS25}.        
\end{remark}

Let $\{a_n\}_{n\in\N}$ be a sequence satisfying Definition \ref{defzeta} and let $k\in\N$ be the positive integer such that $k-1\leq \alpha<k$. 
By the Hadamard factorization theorem \cite[Sec.~13.9 and 13.11]{Nev07} the product
\begin{equation}\label{c1}
   H(z)=\prod_{n=1}^{\infty}E\left(\frac{z}{a_n},k-1\right),\quad \text{ with }\quad 
  E\left(\frac{z}{a_n},k-1\right)=\left(1-\frac{z}{a_n}\right)\exp\left[\sum_{j=1}^{k-1}\frac{1}{j}\left(\frac{z}{a_n}\right)^{j}\right],  
\end{equation}
defines an entire function of $z\in\C$ with roots only at the points $\{a_n\}_{n\in\N}$. We will refer to $H(z)$ as the \textit{Hadamard characteristic function} associated with the sequence $\{a_n\}_{n\in\N}$ and to its expression on the left-hand side of \eqref{c1} as its Hadamard factorization. We will use the expression \textit{characteristic function} for any entire function $F$ with the same zero set $S$ counting multiplicities. Corollary \ref{c2.6} highlights some advantages of using the Hadamard form to compute the values of the associated $\zeta$ at the integers greater than $\alpha$ while formula \eqref{2.30aa} underscores its merits in the process of deforming the complex contour integral representation of the associated $\zeta$-function. In this work, we will assume that all characteristic functions are of minimal order, see Remark \ref{rem2.3}.

The logarithmic derivative of $H(z)$, that is $\frac{d}{dz}\ln\,H(z)$, is then a  meromorphic function in $\C$ possessing poles exactly at the points of the sequence $\{a_n\}_{n\in\bbN}$. The following technical result will be useful later for the construction of an integral representation of the $\z$-function $\z_S(s)$: 
\begin{lemma}\label{l2.2}
    Let $H(z)$ be the function defined in \eqref{c1}. Then for each real $p$, with $\alpha< p\leq k = \lfloor \alpha \rfloor + 1$ and $\delta>0$ there exists an $A(\delta, p) > 0$ such that
    \begin{align}\label{C_eps_p}
        \Big| \frac{d}{dz} \ln \, H(z) \Big| \leq A(\delta,p) |z|^{p-1},
    \end{align}
    for $z\in\C\backslash\mathcal{B}_\delta$ with $\mathcal{B}_\delta$ given in \eqref{2.11} $($see also Fig.~\ref{Fig:Contour}$)$.
\end{lemma}
\begin{proof}
Note that from \eqref{c1} we obtain the following formula for the logarithmic derivative of $H$:

\begin{align}\label{logDer}
    - \frac{d}{dz} \ln \, H(z) = z^{k-1}\sum_{n = 1}^\infty \frac{1}{a_n - z} \frac{1}{a_n^{k-1}}.
\end{align}
Recall that $\sum_{n = 1}^\infty |a_n|^{-k}$ converges by assumption, implying that the above sum converges locally uniformly for $z \in \C \setminus \lbrace a_1, a_2, \dots \rbrace$. We can bound the growth of the logarithmic derivative of $H$ using a standard argument. Choose any $r \in (0,1)$ and separate the infinite sum in \eqref{logDer} as follows:
\begin{align}
    \sum_{n = 1}^\infty \frac{1}{a_n - z} \frac{1}{a_n^{k-1}} = \sum_{n \colon |z/a_n| < r} \frac{1}{1 - z/a_n} \frac{1}{a_n^{k}} \ +  \sum_{n \colon |z/a_n| \geq r} \frac{1}{a_n - z} \frac{1}{a_n^{k-1}}. 
\end{align}
The first sum on the right-hand side has infinitely many terms and can be bounded as 
\begin{align}
    \Big|\sum_{n \colon |z/a_n| < r} \frac{1}{1 - z/a_n} \frac{1}{a_n^{k}}\Big|\leq\frac{1}{1-r}\sum_{n \colon |z/a_n| < r} |a_n|^{-k}.
\end{align}
However, since $|a_n|>r^{-1}|z|$ and $p\leq k$ we have
\begin{align}\label{2.9}
    |a_n|^{-k}=|a_n|^{-p}|a_n|^{p-k}\leq r^{k-p}|a_n|^{-p}|z|^{p-k}, 
\end{align}
and, due to the converge of the series $\sum_{n=1}^\infty |a_n|^{-p}$, there exists a constant $C>0$ such that
\begin{align}\label{2.34}
   \Big| \sum_{n \colon |z/a_n| < r} \frac{1}{1 - z/a_n} \frac{1}{a_n^{k}}\Big| \leq C\frac{r^{k-p}}{1-r}|z|^{p-k}.
\end{align}
Note that this bound holds with no restrictions on $z$. 

In the second sum, since $|z/a_n| \geq r$, it becomes possible to find values of $a_n$ such that $z=a_n$. To avoid this case, we introduce the region 
\begin{align}\label{2.11}
  \mathcal{B}_\delta=\bigcup_{n\in\N}B_{\delta}(a_n),\quad\textrm{with}\quad B_{\delta}(a_n)=\{z\in\C \ \colon |a_n-z|\leq \delta|a_n|\},  
\end{align}
and we consider the restriction $z\in\C\backslash\mathcal{B}_\delta$. We claim that for each $\delta > 0$ there exists a $\delta'> 0$, independent of $a_n$, such that $|a_n - z| \leq \delta'|z|$ implies $|a_n - z| \leq \delta |a_n|$, i.e., $z \in B_\delta(a_n)$. In fact
\begin{align*}
    |a_n - z| \leq \delta' |z| \leq \delta'(|a_n| + |a_n - z|) \quad \Rightarrow \quad |a_n - z| \leq \delta'/(1-\delta')|a_n|.
\end{align*}
In particular, for $z \not \in B_\delta(a_n)$, there exists a $\delta'$ such that $|a_n - z| > \delta'|z|$. It follows that for $z \not \in \mathcal B_\delta$ 
\begin{align}
    \Big|\sum_{n \colon |z/a_n| \geq r} \frac{1}{a_n - z} \frac{1}{a_n^{k-1}}\Big|\leq\sum_{n \colon |z/a_n| \geq r}\frac{1}{\delta'|z||a_n|^{k-1}}.
\end{align}
This time as $|z| \geq r |a_n|$ and $p-k+1 > 0$ 
we have (cf.~\eqref{2.9})
\begin{align*}
    |z|^{-1}|a_n|^{-k+1} = |z|^{-1} |a_n|^{-p} |a_n|^{p-k+1} \leq |z|^{p-k} |a_n|^{-p} r^{k-p-1}
\end{align*}
Since $\sum_{n \colon |z/a_n| \geq r}|a_n|^{-p}<C_1$, we conclude that 
\begin{align}\label{2.37}
    \Big|\sum_{n \colon |a_n/z| \geq r} \frac{1}{a_n - z} \frac{1}{a_n^{k-1}}\Big|\leq C_1\frac{ r^{k-p-1}}{\delta'}|z|^{p-k}.
\end{align}

The results \eqref{2.34} and \eqref{2.37} allow us to conclude that for some $A(\delta, p)>0$
 \begin{align}\label{Est_F1}
     \left| \frac{d}{dz} \ln \, H(z) \right| \leq A(\delta,p) |z|^{p-1}, 
 \end{align}
in the region $z\in\C\backslash\mathcal{B}_\delta$.   
\end{proof}
The logarithmic derivative of $H(z)$ is meromorphic and, according to Lemma \ref{l2.2}, it is polynomially bounded in the region $\C\backslash\mathcal{B}_\delta$, that is as long as $z$ does not get {\it too close} to the points ${a_n}$. In particular, provided condition $(iii)$ of Definition \ref{defzeta} holds, we can choose for any $\varepsilon' < \varepsilon$ a $\delta$ sufficiently small, such that \eqref{Est_F1} holds in $\C \setminus \mathcal{I}_{\varepsilon'}$ (see \eqref{I_eps}). By making $\varepsilon$ smaller if necessary, we can assume that $\varepsilon' = \varepsilon$.   

Lemma \ref{l2.2} is therefore crucial for proving the following contour integral representation of $\z_S(s)$.

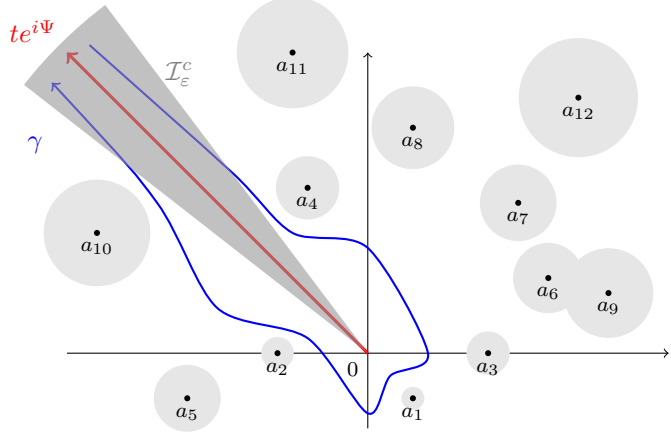
\begin{figure}
\begin{tikzpicture}[scale=2]

\draw[->] (-2,0) -- (2,0);
\draw[->] (0,-0.5) -- (0,2);

\draw[very thick, red, ->] (0,0) -- (-2,2) node[above left] {$t e^{i\Psi}$};

\draw[blue, thick, <-]
    plot[smooth] coordinates {
        (-2.1, 1.8)
        (-1.4, 1)
        (-1,0.3)
        (-0.4, 0.1)
        (0, -0.4)
        (0.15,-0.15)
        (0.4, 0)
        (0, 0.7)
        (-0.5, 0.8)
        (-0.9, 1.2)
        (-1.85, 2.05)
    };

\foreach \x/\y in {-1.8/0.8, 0.3/-0.3, -0.4/1.1, 1.4/1.7, 0.8/0, 1/1, 0.3/1.5, 1.6/0.4, 1.2/0.5, -0.6/0, -1.2/-0.3, 1.6/0.4, -0.5/2} {
    \pgfmathsetmacro{\r}{0.18*sqrt(\x*\x + \y*\y)} 
    \fill[gray!20] (\x,\y) circle (\r);
    \fill[black] (\x,\y) circle (0.02);
}

\node at (0.3,-0.3) [below] {\footnotesize{$a_1$}};
\node at (-0.6,0) [below] {\footnotesize{$a_2$}};
\node at (0.8,0) [below] {\footnotesize{$a_3$}};
\node at (-0.4,1.1) [below] {\footnotesize{$a_4$}};
\node at (-1.2,-0.3) [below] {\footnotesize{$a_5$}};
\node at (1.2,0.5) [below] {\footnotesize{$a_6$}};
\node at (1,1) [below] {\footnotesize{$a_7$}};
\node at (0.3,1.5) [below] {\footnotesize{$a_8$}};
\node at (1.6,0.4) [below] {\footnotesize{$a_9$}};
\node at (-1.8,0.8) [below] {\footnotesize{$a_{10}$}};
\node at (-0.5,2) [below] {\footnotesize{$a_{11}$}};
\node at (1.4,1.7) [below] {\footnotesize{$a_{12}$}};

\node at (0,0) [below left] {\footnotesize{$0$}};

\node at (-2.1,1.5) [below left] {\textcolor{blue}{$\gamma$}};

\node at (-1.1,2) [below left] {\textcolor{gray}{$\mathcal I_\varepsilon^c$}};

\fill[gray!80, opacity=0.6] (0,0) 
    -- ({2.9*cos(127)}, {2.9*sin(127)}) 
    arc[start angle=127, end angle=142, radius=2.9cm] 
    -- cycle;

\end{tikzpicture}
\caption{Illustration of a sequence $\{a_n\}_{n\in\bbN}$ as black dots (with $\mathcal{B}_\delta (a_n)$ the light gray shaded regions around each $a_n$), the complement $\mathcal{I}_\varepsilon^c$ given as the dark gray sector, the branch cut $te^{i \Psi}$ in red, and the contour $\gamma$ in blue.}
\label{Fig:Contour}
\end{figure}

\begin{theorem}\label{Theorem_Integral}
Let $\z_{S}(s)$ be the $\z$-function of a sequence $S$ satisfying $(i)$, $(ii)$ in Definition \ref{defzeta}, such that $\Arg(a_n) \not = \Psi \in [-\pi, \pi)$ and the branch cut for $a \mapsto a^{-s}$ is chosen on the ray $R_\Psi=\{z=te^{i\Psi}:t\in [0,\infty)\}$. Let $H(z)$ be the characteristic function, \eqref{c1}, of $S$. Assume that $\gamma$ is a contour belonging to $\C\backslash\mathcal{B}_\delta$, for some $\delta>0$ small enough, that avoids the branch cut $R_\Psi$ and encloses, in the counterclockwise direction, all the points of the sequence $S$ $($see Fig.~\ref{Fig:Contour}$)$. Then for $\Re(s)>\alpha$,
\begin{align}\label{2.17}
    \zeta_S(s)=\frac{1}{2\pi i}\int_\gamma dz \, z^{-s} \frac{d}{dz} \ln \, H(z) .
\end{align}
\end{theorem}
\begin{proof}
Let us introduce the truncated characteristic function   \begin{align}\label{2.18}
    H_N(z) = \prod_{n=1}^N E\left(\frac{z}{a_n},k-1\right).
\end{align}  
The result of Lemma \ref{l2.2} holds for $H_N$ with virtually the same proof and hence
\begin{align}\label{2.20}
    \Big| -\frac{d}{dz} \ln \, H_N(z) \Big| \leq A(\delta, p) |z|^{p-1},
\end{align}
for $z\in\C\backslash\mathcal{B}_\delta$. In particular, we can, without loss of generality, choose the same constant $A(\delta, p) > 0$ as in \eqref{C_eps_p}, as we now have to bound the truncated sum instead of the infinite sum in \eqref{logDer}.
As $\gamma$ is contained in the region $\C\backslash\mathcal{B}_\delta$ by assumption, we conclude that
\begin{align}\label{Est_FN}
     \Big|z^{-s}\frac{d}{dz} \ln \, H_N(z) \Big| \leq  A(\delta, p) e^{|2\pi \Im(s)|}|z|^{p-1-\Re(s)}, \qquad z \in \gamma.
\end{align}
In particular, for $\Re(s)>p$ the integral 
\begin{align}
    \frac{1}{2\pi i}\int_\gamma dz \, z^{-s} \frac{d}{dz} \ln \, H_N(z) ,
\end{align}
converges and is equal to the truncated sum $\sum_{n=1}^N a_n^{-s}$ by Cauchy's theorem.

Note that the inequality \eqref{Est_FN} remains true upon substituting $H$ for $H_N$ with the same constant $A(\delta, p)$ (this is just the content of \eqref{Est_F1}).
The key observation is that the right-hand side of \eqref{Est_FN} is a dominating function for the sequence of functions $z^{-s}\frac{d}{dz} \ln \, H_N(z)|_{z \in \gamma}$ which converges pointwise to $z^{-s}\frac{d}{dz} \ln \, H(z)|_{z \in \gamma}$ as $N \to \infty$. Using dominated convergence we conclude that
\begin{align}
    \lim_{N\to \infty} \underbrace{\frac{1}{2\pi i}\int_\gamma dz \, z^{-s} \frac{d}{dz} \ln \, H_N(z)}_{= \ \sum_{n=1}^N a_n^{-s}} = \frac{1}{2\pi i}\int_\gamma dz \, z^{-s} \frac{d}{dz} \ln \, H(z), \quad \Re(s) > p.
\end{align}
But clearly the left-hand side converges to $\zeta_S(s)$. Hence, we arrive at the following formula
\begin{align}
    \zeta_S(s)=\frac{1}{2\pi i}\int_\gamma dz \, z^{-s} \frac{d}{dz} \ln \, H(z), \quad \text{for } \, \Re(s) > p.
\end{align}
Finally, as this formula holds for  $\Re(s) > p$ with $p > \alpha$, it also holds for any $s$ satisfying $\Re(s) > \alpha$. 
\end{proof}


\begin{remark}\label{rem2.3}
    Any other entire function $F(z)$ with the same zeros as $H(z)$ in \eqref{c1} and of the same order can be written as \cite[Sec.~13.9]{Nev07} 
  \begin{align}
      F(z)=e^{\omega(z)} H(z),
  \end{align}
where $\omega(z)$ is a polynomial with $\deg(\omega) \leq \lfloor \alpha \rfloor$. With the assumptions of Theorem \ref{Theorem_Integral} (in particular $\Re(s) > \alpha$), we obtain
\begin{align}
    \frac{1}{2\pi i} \int_\gamma dz  \, z^{-s} \frac{d}{dz} \ln \, F(z) &= \frac{1}{2\pi i} \int_\gamma dz  \, z^{-s} \frac{d}{dz} \ln \, H(z) + \frac{1}{2\pi i} \int_\gamma dz  \, z^{-s} \frac{d}{dz} \omega(z)= \zeta_S(s).
\end{align}
Here, the second integral evaluates to $0$, as it converges for $\Re(s) > \alpha$ and the integrand is a holomorphic function in the interior of $\gamma$. This shows that the particular choice of characteristic function does not matter, so long as the order is minimal. 
\end{remark}

\subsection{Deformation of contour integral representation and positive integer values}
Next, we describe a general deformation process that extends the standard procedure utilized in the literature (see, e.g., \cite{FGKS21, FPS25}) to make the structure of the analytic continuation of the $\zeta$-function more apparent. In particular, we prove the following: 
\begin{theorem}\label{t2.7}
Suppose the characteristic function, $F(z)$, is entire with zeros at the sequence $S=\{a_n\}_{n\in\bbN}$ $($satisfying Definition \ref{defzeta}$)$ and minimal growth order $($as describe in Remark \ref{rem2.3}$)$. Let $C_R$ be a clockwise circle with radius $R$ parametrized via $z=R e^{i\theta}$ from $\theta=\Psi$ to $\theta=\Psi-2\pi$. Then the following integral representation holds$:$
\begin{equation}
\zeta_S(s)=e^{is(\pi-\Psi)}\frac{\sin(\pi s)}{\pi}\int_{R}^{\infty}dt\,t^{-s}\frac{d}{dt}\ln\,F\left(te^{i\Psi}\right)-\frac{1}{2\pi i}\int_{C_R}dz\,z^{-s}\frac{d}{dz}\ln\,F\left(z\right),\quad \Re(s)>\alpha,\label{2.28a}
\end{equation}
for $R<|a_1|$ where $a_1$ is the element of $S$ with the smallest modulus.

Moreover, if $\alpha<1$, then
\begin{equation}\label{2.30a}
\z_S(s)=e^{is(\pi-\Psi)}\frac{\sin(\pi s)}{\pi}\int_0^{\infty}dt\,t^{-s}\frac{d}{dt}\ln\,F\left(te^{i\Psi}\right),\quad \alpha<\Re(s)<1.
\end{equation}

If $F(z)$ is further assumed to be given by the Hadamard factorization \eqref{c1}, then
\begin{equation}\label{2.30aa}
\z_S(s)=e^{is(\pi-\Psi)}\frac{\sin(\pi s)}{\pi}\int_0^{\infty}dt\,t^{-s}\frac{d}{dt}\ln\,F\left(te^{i\Psi}\right),\quad \alpha<\Re(s)<\lfloor\a\rfloor+1.
\end{equation}
\end{theorem}
\begin{proof}
Note that by Remark \ref{rem2.3}, we can consider a characteristic function, $F(z)$, with zeros $\{a_n\}_{n\in\bbN}$ and minimal growth order while still applying the estimate given in Lemma \ref{l2.2} to deform the contour to the branch cut for large-$|z|$ since the integrand in \eqref{2.17} remains polynomially bounded throughout. Hence, we proceed by deforming $\gamma$ into a contour that consists of three pieces: a portion hugging the branch cut from above parameterized by $z=t e^{i\Psi}$ from $t=\infty$ to $t=R$ (with $R<|a_1|$ where $a_1$ is the element $S$ with the smallest modulus), the circle $C_R$, and a portion hugging the branch cut from below parameterized by $z=t e^{i(\Psi-2\pi)}$ from $t=R$ to $t=\infty$. After the deformation
one obtains 
\begin{align}
\zeta_S(s)&=-\frac{e^{-is\Psi}}{2\pi i}\int_{R}^{\infty}dt\,t^{-s}\frac{d}{dt}\ln\,F\left(te^{i\Psi}\right)+\frac{e^{-is(\Psi-2\pi)}}{2\pi i}\int_{R}^{\infty}dt\,t^{-s}\frac{d}{dt}\ln\,F\left(te^{i(\Psi-2\pi)}\right)\nonumber\\
&\quad-\frac{1}{2\pi i}\int_{C_R}dz\,z^{-s}\frac{d}{dz}\ln\,F\left(z\right) \notag \\
&=e^{is(\pi-\Psi)}\frac{\sin(\pi s)}{\pi}\int_{R}^{\infty}dt\,t^{-s}\frac{d}{dt}\ln\,F\left(te^{i\Psi}\right)-\frac{1}{2\pi i}\int_{C_R}dz\,z^{-s}\frac{d}{dz}\ln\,F\left(z\right).
\end{align} 

We next consider the restriction on $\alpha$ to further deform the integral over the circle of radius $R$ in \eqref{2.28a} to hug the branch cut. By noticing that $F(z)$ is an entire function with $F(0)\neq0$, one concludes that, for any characteristic function with minimal growth order,
\begin{equation}\label{2.40}
\left|\frac{d}{dz}\ln\,F(z)\right|=C+O(|z|) ,\quad |z|\to 0,
\end{equation}  
for some $C\in\bbR$, yielding
\begin{equation}\label{2.30}
-\frac{1}{2\pi i}\int_{C_R}dz\,z^{-s}\frac{d}{dz}\ln\,F\left(z\right)=e^{is(\pi-\Psi)}\frac{\sin(\pi s)}{\pi}\int_0^{R}dt\,t^{-s}\frac{d}{dt}\ln\,F\left(te^{i\Psi}\right),\quad \Re(s)<1.
\end{equation}
Therefore, if $\alpha<1$, one obtains \eqref{2.30a} from \eqref{2.28a} and \eqref{2.30}. 

Moreover, notice that if $F(z)$ is given by the Hadamard factorization \eqref{c1}, then 
$(d/dz)\ln\, F(z)=c_{\lfloor\a\rfloor+1}(\lfloor\a\rfloor+1)z^{\lfloor\a\rfloor}+O(z^{\lfloor\a\rfloor+1})$ as, letting $k=\lfloor\a\rfloor+1$, the series expansion can be written as 
\begin{equation}\label{serieseq}
F(z)=1+\sum_{j=k}^\infty c_jz^j,
\end{equation}
since for $E$ in \eqref{c1} near $z=0$ we have $E(z, k-1) = (1-z)e^{-\log(1-z)+O(z^k)} = e^{O(z^k)} = 1 + O(z^k)$. Hence \eqref{2.40} can be further refined and the integral along the circle can be shrunk to the branch cut for $\Re(s)<1+\lfloor\a\rfloor$. Therefore, if $F(z)$ is given by the Hadamard factorization, the integral representation \eqref{2.30a} will always be valid for $\Re(s)\in(\alpha,\lfloor\a\rfloor+1)$.
\end{proof}
\begin{remark}
$(i)$ The proof of Theorem \ref{t2.7} employs the branch cut $R_{\Psi}$ to obtain the expression \eqref{2.28a}.
One could drop item $(iii)$ from Definition \ref{defzeta} and consider a more general branch cut curve connecting 0 to $\infty$ in $\C\backslash\mathcal{B}_{\delta}$. In this case the result of Theorem \ref{t2.7} is still valid but with the first integral in \eqref{2.28a} suitably modified by the parametrization of the chosen branch cut curve.\\[1mm]
$(ii)$ We point out that, if $\alpha\geq1$ and $F(z)$ is not assumed to be given by the Hadamard factorization \eqref{c1}, the integral over the circle in \eqref{2.28a} cannot in general be shrunk to the branch cut until the first integral in \eqref{2.28a} is analytically continued to the left of $\Re(s)=1$. The procedure of deforming the contour to completely hug the branch cut (as in \eqref{2.30a} and \eqref{2.30aa}) represents the first step of the standard analytic continuation of the spectral $\z$-function performed, for instance, in \cite{FGKS21, FPS25}, and holds, in general, for any characteristic function with minimal growth order if and only if $\alpha<1$. We will return to this representation in Remark \ref{remposint}.
\end{remark}

The following theorem regarding the $\zeta$-function at positive integer values follows readily from the integral representation \eqref{2.28a}. We emphasize  that Theorem \ref{t4} yields universal sum rules, valid for $\zeta$-functions associated with a sequence of complex numbers, in the form of a recurrence relation. That is, the value of the $\zeta$-function at positive integers (larger than $\alpha$) is written as a linear combination of the value of the $\zeta$-function at the previous admissible positive integers. As mentioned in the introduction, this result can be understood as the complete generalization of the classical sum rules known to be satisfied by certain sequences of numbers such as the inverse squares of Bessel function zeros (see, e.g., \cite{AB96}, \cite[Sec.~7.9]{EMOT53II}, \cite{Ra74,Sn60}). Other sum rule representations are discussed after the theorem.
\begin{theorem}\label{t4}
Given a characteristic function, $F(z)$, with zeros coinciding with the elements of the sequence $\{a_n\}_{n\in\bbN}$ $($satisfying Definition \ref{defzeta}$)$, minimal growth order $($as described in Remark \ref{rem2.3}$)$, and series expansion $F(z)=\sum_{j=0}^\infty c_jz^j,$ the following holds$:$
\begin{align}\label{2.34aa}
    \zeta_S(n)=-\Res \left[ z^{-n}\dfrac{d}{dz}\ln\,F(z);\ z=0 \right]=-n b_n,\quad n\in\mathbb{N},\ n>\alpha,
\end{align}
where
\begin{equation}\label{t2.28}
        b_1=\frac{c_{1}}{c_0},\quad
        b_j=\frac{c_{j}}{c_0}-\sum_{\ell=1}^{j-1}\left(\dfrac{\ell}{j}\right)\frac{c_{j-\ell}}{c_0} b_{\ell},\quad j\geq 2,
\end{equation}
implying 
\begin{align}
\zeta_S(n)=-n\frac{c_{n}}{c_0} -\sum_{\ell=\lfloor \alpha\rfloor+1}^{n-1}\frac{c_{n-\ell}}{c_0}\zeta_S(\ell)+\sum_{\ell=1}^{\lfloor \alpha\rfloor}\ell \frac{c_{n-\ell}}{c_0}b_{\ell},\quad n\in\mathbb{N},\ n>\alpha.
\end{align}

In particular, if $\alpha<1$ then 
\begin{equation}\label{t2.25}
\zeta_S(n)=-n\frac{c_{n}}{c_0} -\sum_{\ell=1}^{n-1}\frac{c_{n-\ell}}{c_0}\zeta_S(\ell),\quad n\in\mathbb{N}.
\end{equation}
\end{theorem}
 \begin{proof}
The proof is analogous to that of \cite[Thm. 4.1]{FGKS21} (see also \cite[Thm. 3.7]{FPS25}) noting that the branch cut $R_\Psi$ is no longer needed whenever $s=n\in\mathbb{N}$ so that \eqref{2.28a} reduces to a clockwise circle about the origin whenever $n>\alpha$. This is due to the fact that the first integral in \eqref{2.28a} converges for $\Re(s)>\a$ and vanishes at $s=n>\a$ since it is multiplied by the prefactor $\sin(\pi s)$.
\end{proof}

\begin{remark}\label{r2.10}
$(i)$ As previously observed, we could repeat the proof of Theorem \ref{t4} without item $(iii)$ of Definition \ref{defzeta}. Since the argument of the proof does not depend on the branch cut, the results of Theorem \ref{t4} would remain completely unchanged.\\[1mm]
$(ii)$ We would like to point out that another representation for $\zeta_S(n)$ is given by
\begin{equation}
\zeta_S(n)=-n\sum_{k=1}^{n}\frac{(-1)^{k-1}}{k c_0^k} B_{n,k}(c_1,c_2,\dots,c_{n-k+1}),\quad n\in\bbN,\ n>\alpha,
\end{equation}
where $F(z)=\sum_{j=0}^\infty c_jz^j$ as in Theorem \ref{t4} and $B_{n,k}$ are \textit{ordinary Bell polynomials} \cite[p. 136]{Co74},
\begin{equation}
B_{m,k}(c_1,c_2,\dots,c_{m-k+1})=\sum \frac{k!}{j_1!j_2!\dots j_{m-k+1}!}c_1^{j_1} c_2^{j_2} \dots c_{m-k+1}^{j_{m-k+1}},
\end{equation}
with the sum being over all sequences $j_1,j_2,\dots,j_{m-k+1}$ of nonnegative integers such that
\begin{align}
j_1+j_2+\dots+j_{m-k+1}=k,\quad j_1+2j_2+\dots+(m-k+1)j_{m-k+1}=m.
\end{align}
This representation follows by utilizing the generating function for these polynomials to write
\begin{equation}
\bigg(\sum_{j=1}^\infty c_j z^j\bigg)^k=\sum_{m=k}^\infty B_{m,k}(c_1,c_2,\dots,c_{m-k+1})z^m,
\end{equation}
so that the series for $(d/dz)\ln\, F(z)$ is given by
\begin{align}
\dfrac{d}{dz}\ln\, F(z)&=\dfrac{d}{dz}\ln\bigg(1+c_0^{-1}\sum_{j=1}^\infty c_j z^j\bigg)=\frac{d}{dz}\sum_{k=1}^\infty \frac{(-1)^{k-1}}{k c_0^k}\bigg(\sum_{j=1}^\infty c_j z^j\bigg)^k\notag\\
&=\sum_{k=1}^\infty \frac{(-1)^{k-1}}{k c_0^k}\sum_{m=k}^\infty m B_{m,k}(c_1,c_2,\dots,c_{m-k+1})z^{m-1}.
\end{align}
Collecting the coefficients of the $m=n$ terms will now give the residue needed in Theorem \ref{t4}.\\[1mm]
$(iii)$ Other sum rules that have garnered much attention in the mathematical physics literature are the so-called `exact sum rules' (see \cite{Ka25,voros23} and the references therein).  To simplify notation, let us write $\hat c_n = c_n / c_0$ and $\hat F(z)  = F(z)/c_0$. We trivially have the formula $\hat F(z) = \exp(\ln \,   \hat F(z) )$, from which comparing Taylor coefficients on both sides we obtain the identity
\begin{equation}
    \hat c_n = \sum_{j_1 + \, \dots \, + j_k = n} \frac{1}{k!} b_{j_1} \cdots b_{j_k} = b_n + \sum_{j_1 + \, \dots \, + j_k = n, \ k \geq 2} \frac{1}{k!} b_{j_1} \cdots b_{j_k} 
\end{equation}
This yields the \textit{almost exact sum rule} 
\begin{equation}\label{2.44a}
    \zeta_S(n) = - n\frac{c_n}{c_0} + \sum_{j_1 + \, \dots \, + j_k = n, \ k \geq 2} (-1)^{\#\{j_\ell>\alpha\}}\frac{n}{k!}\prod_{j_\ell\leq\a} b_{j_\ell} \prod_{j_\ell>\a} \frac{\zeta_S(j_\ell)}{j_\ell},\quad n>\a,
\end{equation}
where $\#\{j_\ell>\alpha\}$ denotes the number of $j_{\ell}$ larger than $\a$ and each of the $b_{j_\ell}$ can be replaced by $b_{j_\ell}=-j_\ell^{-1} \zeta_S(j_\ell)$ only for $j_\ell>\a$ in general, hence the terminology `almost' exact sum rule.
Thus, it becomes apparent that to obtain a proper exact sum rule, it is \emph{sufficient} (though not necessary of course) for all the points $n\leq \a$, $n\in\bbN$, to be regular points of $\zeta_S$ as well as to express $b_{j_{\ell}}$ in terms of $\zeta_S(j_\ell)$ for $j_\ell\leq\a$ and $\hat c_n = c_n/c_0$ in terms of a sum of products each satisfying the same condition as above. As we have imposed minimal assumptions on the complex sequence $S$, \eqref{2.44a} is the best one can expect in general. We would like to point out that there are sequences $S$ and characteristic functions $F(z)$ such that Theorem \ref{t4}, specifically \eqref{2.34aa}, holds for all $n\in\bbN$ regardless of the value of $\a$ (see Remark \ref{remposint} $(i)$ and the Airy $\zeta$-function in Section \ref{s5}).

If we now assume that $c_1\neq0$ and either $\a<1$ or we are in the setting such that \eqref{2.34aa} holds for all $n\in\bbN$ described in Remark \ref{remposint} $(i)$, a somewhat canonical way of achieving a proper exact sum rule is artificially inserting $\zeta_S^n(1)$ into the first term by writing
\begin{equation}
\hat c_n=\frac{c_n}{c_0}=(-1)^{n}\frac{c_0^{n-1}c_n}{c_1^{n}} \zeta_S^n(1),\quad n\in\bbN,
\end{equation}
to obtain the \textit{universal exact sum rule}
\begin{equation}\label{2.46aa}
    \zeta_S(n) = (-1)^n n\bigg(\frac{1}{n!}-\frac{c_0^{n-1}c_n}{c_1^{n}}\bigg)\zeta_S^n(1) + \sum_{j_1 + \, \dots \, + j_k = n, \ n>k \geq 2} \frac{(-1)^kn}{k!j_1 \cdots j_k} \zeta_S(j_1) \cdots \zeta_S(j_k),\quad n\in\bbN. 
\end{equation}
Note that exact sum rules have typically been defined with real-valued coefficients (since the previous papers consider underlying sequences which are real-valued), however the coefficient of $\zeta_S^n(1)$ in \eqref{2.46aa} is possibly complex-valued in the general setting we consider here. These exact sum rules have been studied whenever the underlying sequence is the spectrum of a differential operator and, because of this, are also referred to as spectral sum rules \cite{voros23}. Importantly, our work here shows that one universal form for such sum rules is \eqref{2.46aa} (under the assumptions above) regardless of any additional assumptions on the sequence yielding a direct and efficient avenue for studying such spectral problems.

Applying the same trick to the recursion in 
\eqref{t2.25} when \eqref{2.34aa} holds for all $n\in\bbN$ we also obtain
\begin{align}\label{2.38aa}
\zeta_S(n)=(-1)^{n-1}\frac{c_0^{n-1}}{c_1^{n}}\bigg(n c_n \zeta_S^n(1)+\sum_{\ell=1}^{n-1}\Big(-\frac{c_0}{c_1}\Big)^\ell c_{n-\ell}\zeta_S^{n-\ell}(1)\zeta_S(\ell)\bigg),\quad n\in\bbN.
\end{align}
In particular, one has
\begin{align}
&\zeta_S(2)-c_1^{-2}\big(c_1^2-2c_2c_0\big)\zeta_S^2(1)=0,\quad \zeta_S(3)-\zeta_S(2)\zeta_S(1)+c_1^{-3}\big(c_2c_1c_0-3c_3 c_0^2\big)\zeta_S^3(1)=0,\dots
\end{align}

Finally, note that one can obtain further such sum rules from \eqref{2.44a}, \eqref{2.46aa}, and \eqref{2.38aa} by iteratively replacing some $\zeta_S(j_\ell)$ via the same formulas.
\end{remark}

Whenever $F(z)$ in Theorem \ref{t4} is given by the Hadamard factorization \eqref{c1}, the formulas above become much simpler due to the expansion \eqref{serieseq}.
This observation leads to the following:
\begin{corollary}\label{c2.6}
If $F(z)$ is given by the Hadamard factorization \eqref{c1} with series \eqref{serieseq}, one has
\begin{align}
    \zeta_S(n)=-\Res \left[ z^{-n}\dfrac{d}{dz}\ln\,F(z);\ z=0 \right]=-n b_n,\quad n\in\mathbb{N},\ n>\alpha,
\end{align}
where, for $k=\lfloor \a\rfloor+1$,
\begin{equation}\label{t2.39}
        b_k=c_{k}, \dots, b_{2k-1} = c_{2k-1}, \quad
        b_j=c_{j}-\sum_{\ell=k}^{j-k}\left(\dfrac{\ell}{j}\right)c_{j-\ell} b_{\ell},\quad j\geq 2k,
\end{equation}
implying 
\begin{align}
\zeta_S(n) = -nc_n \ \text{ for } \ n =k, \dots, 2k-1, \quad
\zeta_S(n)=-nc_{n} -\sum_{\ell=k}^{n-k}c_{n-\ell}\zeta_S(\ell),\quad n\in\bbN,\ n\geq 2k.
\end{align}
\end{corollary}
\begin{proof}
Follows as the previous proof by utilizing \eqref{serieseq} and the fact that
\begin{equation}
\log\bigg(1+\sum_{j = k}^\infty c_j z^j\bigg) = \sum_{j = k}^\infty c_j z^j - \frac{1}{2}\underbrace{\bigg(\sum_{j = k}^\infty c_j z^j\bigg)^2}_{O(z^{2k})}+\dots,
\end{equation}
implies $b_j = c_j$ for all $j$ up to and including $j = 2k-1$.
\end{proof}

Since the Hadamard factorization $H(z)$ plays a special role in the integral representation of the $\z$-function, we would like to provide a method of finding such a Hadamard factorization of the same order given a characteristic function $F(z)$. Let $F(z)$ be an entire function with roots at the points of the sequence $S$, then according to Remark \ref{rem2.3} one can write $ F(z)=e^{\omega(z)} H(z)$ with $\omega(z)$ a polynomial such that $\textrm{deg}(\omega)\leq\lfloor\alpha\rfloor$ and $H(z)$ the corresponding Hadamard factorization of the same order as $F(z)$. Then
\begin{align}\label{2.42a}
    \ln\,F(z)=C_{\lfloor\alpha\rfloor}z^{\lfloor\alpha\rfloor}+C_{\lfloor\alpha\rfloor-1}z^{\lfloor\alpha\rfloor-1}+\cdots+C_1z+C_0+\ln\,H(z).
\end{align}
We know that for small values of $|z|$, $F(z)=\sum_{j=0}^\infty c_jz^j,$ and hence
\begin{align}
    \ln\,F(z)=\ln\,c_0+\sum_{n=1}^{\infty}b_{n}z^{n},
\end{align}
with the coefficients $b_{n}$ given in \eqref{t2.28}. On the other hand, by \eqref{serieseq} one has $H(z)=1+\sum_{j=\lfloor\alpha\rfloor+1}^\infty d_jz^j$ and therefore
\begin{align}
    \ln\,H(z)=\sum_{m=\lfloor\alpha\rfloor+1}B_{m}z^{m}, 
\end{align}
with the coefficients $B_{m}$ given by \eqref{t2.39}.
By substituting the last two expressions in \eqref{2.42a} and comparing the coefficients of the series on both sides of the ensuing equation we obtain
\begin{align}
    C_{0}=\ln\,c_0,\quad C_{1}=b_1, \ \ldots, C_{\lfloor\alpha\rfloor}=b_{\lfloor\alpha\rfloor},\quad B_{m}=b_{m},\quad m\geq\lfloor\alpha\rfloor+1.
\end{align}
This implies that given a characteristic function $F(z)$ of order $\alpha$ with prescribed zeros, the corresponding Hadamard factorization $H(z)$ of the same order is 
\begin{align}
    H(z)=c_0^{-1}F(z)\exp\left\{-\sum_{m=1}^{\lfloor\alpha\rfloor}b_{m}z^{m}\right\}.
\end{align}

\begin{remark}\label{remposint}
$(i)$ As previously mentioned, while formulas of the form \eqref{2.30a} and \eqref{2.30aa} have often been used in the literature, the representation of the $\z$-function in \eqref{2.28a} is actually quite illuminating. We recall that if we set $s=n>\a$, $n\in\bbN$, the branch cut is no longer needed so that the first integral in \eqref{2.28a} converges and vanishes due to the $\sin(\pi s)$ prefactor. The second integral is now simply an integral around a small circle about the origin, yielding Theorem \ref{t4}.

More importantly, we would like to add that the second integral in \eqref{2.28a} always reduces to this form for any $s=n$, $n\in\bbN$, and therefore it always contributes with a factor $-nb_n$ (see \eqref{t2.28}) to the potential value of $\zeta_S(n)$ for all $n\in\bbN$, regardless of the value of $\alpha$ (note, however, that this second integral is zero for $s=-n$, $n\in\bbN_0$). Now, whether $\zeta_S(n)$ for $n\leq\alpha$ is singular or not (and its exact value if it is, in fact, regular), will depend also on the large-$z$ asymptotics of $\frac{d}{dt}\ln\, F(z)$ appearing in the integrand of the first integral in \eqref{2.28a}. This dependence manifests itself through the analytic continuation described in the next section.

It is, then, interesting to note that it is possible, after a proper analytic continuation, for the first integral in \eqref{2.28a} to give no contribution to the value of $\z_{S}(n)$ for $1\leq n\leq\alpha$ leaving only the one coming from the second integral in \eqref{2.28a}. 
This possibility is absolutely fascinating since, in this case, Theorem \ref{t4} would in fact hold for all $n\in\bbN$, regardless of the value of $\alpha$ needed for the convergence of the sum! This is indeed the case for the Airy $\zeta$-function analyzed in Section \ref{Sect:Airy_Zeros} with $F(z)=\Ai(-z)$. 

If, on the other hand, after a proper analytic continuation, the first integral in \eqref{2.28a} yields a non-vanishing contribution to the value of $\zeta_S(n)$ for $1\leq n\leq\alpha$ (as long as it is a regular point of $\zeta_S(s)$), then one could still extend the results of Theorem \ref{t4} to the range $1\leq n\leq\alpha$. However, those values would have to be modified by the addition of the non-vanishing contribution coming from the analytically continued first integral.
Examples of such behavior are provided by the $\zeta$-functions associated with the zeros of the parabolic cylinder function at $s=2$ and the confluent hypergeometric function at $s=1$ discussed in Section \ref{s5}.
We point out that Theorem \ref{t3.7} illustrates this fact under a natural large-$z$ asymptotic expansion assumption of the logarithm of the characteristic function.\\[1mm]
$(ii)$ As a final intriguing observation, we point out that if $F(z)$ is given by the Hadamard factorization \eqref{c1}, the first non-vanishing term in its series expansion \eqref{serieseq} is the one corresponding to the $\lfloor\a\rfloor+1$ power of $z$. In particular, this implies that $b_j=0$ for $j\in\{1,\dots,\lfloor\a\rfloor\}$ so that, in this case, the second integral in \eqref{2.28a} will always be zero for $n<\alpha$. Therefore when $F(z)$ is given by the Hadamard factorization \eqref{c1}, any nonzero contribution to the value $\zeta_S(n)$ for $n<\a$ will come solely from the analytic continuation of the first integral in \eqref{2.28a}. Yet again this shows the peculiarities that become apparent when working with the Hadamard factorization form.
\end{remark}
\subsection{The connection to regularized Fredholm determinants and traces}\label{s6}
One of the most important applications of the contour integral approach to $\z$-functions can be found in spectral theory. 
In fact, one can define the $\z$-function associated with a self-adjoint operator $\cK$ that is densely defined in the Hilbert space $\cH$, bounded from below, and has a discrete spectrum. The ensuing spectral $\z$-function is well defined since the spectrum of this class of operators satisfies the assumptions of Definition \ref{defzeta}.   
 
In order to construct an integral representation for the spectral $\z$-function, we need to first describe a property of the resolvent operator $(\cK-z I)^{-1}$ which plays an important role in the determination of a suitable function whose zeros coincide with the eigenvalues of $\mathcal{K}$. Let $\cB_{p}(\cH)$, $1\leq p< \infty$, represent the space of $p$-Schatten–von Neumann class operators (consult \cite{Sim05} for a survey on the subject). The following result shows that if the sequence of eigenvalues of $\cK$ has exponent of convergence less than $p$, then $(\cK-z I)^{-1}$ is a $p$-Schatten–von Neumann class operator.  
\begin{lemma}\label{lemma2}
Let $(\cK-z I)^{-1}$ with $z\in\rho(\cK)$ be the resolvent of the self-adjoint operator $\cK$, densely defined in $\cH$, bounded from below, and with discrete spectrum $\lambda_n\in\sigma(\cK)$.
If for $p>0$ the spectrum satisfies
\begin{align}\label{seriesconv}
    \sum_{\underset{\lambda_n \not = 0}{n=1}}^\infty \lambda_n^{-p}<\infty,
\end{align}
and $0$ is an eigenvalue of at most finite multiplicity, then $(\cK-z I)^{-1}\in \cB_{p}(\cH)$ for $z\in\rho(\cK)$.
\end{lemma}
\begin{proof}
It is sufficient to prove that $(\cK-z I)^{-1}\in \cB_{p}(\cH)$ for one particular point in the resolvent of $\cK$ since the same will be true for all $z\in\rho(\cK)$. First we note that since $\cK$ is self-adjoint, $[(\cK-z I)^{-1}]^{\ast}=(\cK-\bar{z} I)^{-1}$. Since $\cK$ is bounded from below, $\rho(\cK)\cap\R\neq \emptyset$, so there exists $x\in\R$ such that $[(\cK-x I)^{-1}]^{\ast}=(\cK-x I)^{-1}$. The spectrum of $(\cK-x I)^{-1}$ is then \cite[Ch.~8]{Dav07}
\begin{align}
    \sigma[(\cK-x I)^{-1}]=\{0\}\cup\{(\lambda_n-x)^{-1}:\lambda_n\in\sigma(\cK)\},
\end{align}
and since $(\cK-x I)^{-1}$ is self-adjoint, it coincides with its singular values. By choosing $x_0\in\rho(\cK)\cap\R$ such that  $x_0<\inf\sigma(\cK)$ we then have that the $p$-Schatten–von Neumann norm 
\begin{align}
    \Norm{(\cK-x_0 I)^{-1}}_{\cB_p}=\left(\sum_{n=1}^\infty[(\lambda_n-x_0)^{-1}]^{p}\right)^{1/p}<\infty,
\end{align}
by virtue of the convergence of the series \eqref{seriesconv}.
\end{proof}   

To construct the spectral $\z$-function of the operator $\cK$ (densely defined in $\cH$, bounded from below, and with a discrete spectrum having exponent of convergence $\alpha$) in terms of an integral representation of the form \eqref{2.17}, we introduce the operator-valued function (see e.g. \cite{GK19} and \cite[Ch.~12]{GNZ23}) 
\begin{align}\label{family}
    A(z)=(z-z_0)(\cK-z_0I)^{-1},\quad z\in\C,\quad z_0\in\rho(\cK).
\end{align}
Let $k\in \N$ be the positive integer so that $k-1\leq \alpha<k$. Since $A(z)\in\cB_{k}(\cH)$ for $z\in\C$, which is an immediate consequence of Lemma \ref{lemma2}, we can express the regularized (modified) Fredholm determinant of the operator $I-A(z)$ as a Weierstrass-type product as follows \cite[Thm. 9.2]{Sim05} 
\begin{align}\label{determinant}
    \textrm{det}_{k}\left(I-A(z)\right)=\prod_{n=1}^{\infty}\left(\frac{\lambda_n-z}{\lambda_n-z_0}\right)\exp\left\{\sum_{j=1}^{k-1}j^{-1}\left(\frac{z-z_0}{\lambda_n-z_0}\right)^{j}\right\}.
\end{align}
The function $\textrm{det}_{k}\left(I-A(z)\right)$ is entire in $z\in\C$ with isolated zeros at the points $\lambda_k\in\sigma(\cK)$ \cite[Ch.~4, Sec.~2]{GK69}. 
The regularized Fredholm determinant in \eqref{determinant} has the same form as the characteristic function $H(z)$ defined in \eqref{c1}. This allows us to utilize the proof of Lemma \ref{l2.2}, with the replacement $a_n\to\lambda_n-z_0$ and $z\to z-z_0$, to obtain the following 
\begin{lemma}\label{l5.2}
  Let $A(z)\in\cB_{k}(\cH)$  be the the family of operators defined in \eqref{family} and let $\textrm{det}_{k}\left(I-A(z)\right)$ be defined as in \eqref{determinant}. Then for each real $p$, with $\alpha<p\leq k$, and $\delta>0$ 
  \begin{align}
      \left|\frac{d}{dz}\ln\,\textup{det}_{k}\left(I-A(z)\right)\right|\leq C(\delta,p)|z-z_0|^{p-1},
  \end{align}
for $z\in\C\backslash \mathcal{B}_{\delta}$ with $\mathcal{B}_{\delta}$ given by \eqref{2.11} with the replacement $a_n\to\lambda_n$.  
\end{lemma}

By using the results of the previous Lemma, one can prove a theorem analogous to Theorem \ref{Theorem_Integral} for the spectral $\z$-function $\zeta_\cK(s)$ (in fact, that proof holds true with the replacement $a_n\to\lambda_n$ and the translation $z\to z-z_0$).
We can therefore conclude that if $\cK$ is a self-adjoint operator which is bounded from below and whose spectrum is discrete and has exponent of convergence $\alpha$, then the spectral $\z$-function associated with $\cK$ has the following integral representation
\begin{align}\label{c10}
    \zeta_\cK(s)= \sum_{n = 1}^\infty \lambda_n^{-s} = \frac{1}{2\pi i}\int_{\gamma}dz\,z^{-s}\left[\frac{d}{dz}\ln\,\textrm{det}_{k}\left(I-A(z)\right)-m_
    0z^{-1}\Big(\frac{z_0-z}{z_0}\Big)^{k-1}\right],
\end{align}
valid for $\Re(s)>\alpha$ where $m_0$ is the multiplicity of the zero eigenvalue and $\gamma$ is a contour that encloses all the eigenvalues $\{\lambda_n\}_{n\in\N}$ in the counterclockwise direction while dipping below, hence not enclosing, the origin. We would like to point out that the logarithmic derivative appearing in the integrand of \eqref{c10} can be written in terms of the trace of a suitable trace-class operator-valued function. In fact, according to \cite[Ch.~1, Sec.~7]{Yaf92} one has
\begin{align}
 \frac{d}{dz}\ln\,\textrm{det}_{k}\left(I-A(z)\right)=-\textrm{tr}_{\cH}\left[(I-A(z))^{-1}A(z)^{k-1}A'(z)\right],   
\end{align}
and by substituting $A(z)$ in \eqref{family} we obtain (cf. \cite[Sec.~12.2]{GNZ23})
\begin{align}
 \frac{d}{dz}\ln\,\textrm{det}_{k}\left(I-A(z)\right)=-\textrm{tr}_{\cH}\left[(\cK-z I)^{-1}A(z)^{k-1}\right].  
\end{align}
This last expression allows us to rewrite the integral representation \eqref{c10} as 
\begin{align}\label{c11}
    \zeta_\cK(s)=-\frac{1}{2\pi i}\int_{\gamma}dz\,z^{-s}\left\{\textrm{tr}_{\cH}\left[(\cK-z I)^{-1}A(z)^{k-1}\right]+m_0z^{-1}\Big(\frac{z_0-z}{z_0}\Big)^{k-1}\right\}.
\end{align}

\begin{remark}
Note that the family of operators $A(z)$ in \eqref{family} depends of the point $z_0\in\rho(\cK)$ which was introduced in order to have a well-defined notion of the determinant of an unbounded operator \cite{GW95}. However, as expected, the integral representation of the spectral $\z$-function in \eqref{c11} does not depend on $z_0\in\rho(\cK)$. This can be seen by considering \eqref{c10} and then applying Remark \ref{rem2.3} to \eqref{determinant} with
\begin{equation}
e^{\omega(z)}=\prod_{n=1}^{\infty}\left(\frac{\lambda_n}{\lambda_n-z_0}\right)\exp\left\{\sum_{j=1}^{k-1}j^{-1}\left[\left(\frac{z-z_0}{\lambda_n-z_0}\right)^{j}-\left(\frac{z}{\lambda_n}\right)^j\right]\right\}.
\end{equation}
Importantly, this illustrates that in the process of  constructing the $\zeta$-function associated with such an operator, one can utilize the simplest form of the characteristic function available rather than the full form of the associated regularized Fredholm determinant. 

For additional results connecting $\zeta$-regularized and Fredholm determinants, we refer to \cite{GK19,Har22}.
\end{remark}

\section{Analytic continuation}\label{continuation}
We now turn to studying the analytic continuation of the $\zeta$-function. From \eqref{2.28a}, it is easy to note that the restriction $\Re(s)>\alpha$ is due to the large-$t$ behavior of the integrand $\ln\,F\left(te^{i\Psi}\right)$ while the integral over the circle of radius $R$ is an entire function of $s\in\C$.  
In order to perform the analytic continuation to a region of the complex plane where $\Re(s)\leq\alpha$, one can subtract, and then add, from the integrand of the first integral in \eqref{2.28a} the large-$t$ asymptotic expansion of $\ln\,F\left(te^{i\Psi}\right)$. This process will move the abscissa of convergence to the left of $\Re(s)=\alpha$. The more terms of the asymptotic expansion are subtracted, the more the abscissa of convergence moves to the left.  

To further outline the process of analytic continuation we need to make some basic assumptions regarding the large-$z$ asymptotic expansion of $\ln\,F(z)$. 
In particular, we assume the following asymptotic expansion which is often used in the spectral $\z$-function literature (see, e.g., \cite[Sec.~2.1]{Sp06}).

\begin{assumption}\label{assump2.9}
Assume that the sequence $\{a_n\}_{n\in\bbN}$ satisfies Definition \ref{defzeta}. Choose a branch of the function $\ln \, F(z)$ which is holomorphic in the sector $\mathcal I_\varepsilon^c$ $($cf.~Fig.~\ref{Fig:Contour}$)$ and assume it satisfies the large-$z$ asymptotic expansion
\begin{equation}\label{2.42}
\ln\,F(z)= \sum_{j=0}^N  \sum_{k=0}^{M} d_{j,k} z^{\alpha-(j/m)}\ln^{k}z+o\big(z^{\alpha-(N/m)-\delta}\big),\quad z\to\infty, \ \ z \in \mathcal I_{\varepsilon}^c,
\end{equation}
for some $m,M,N\in\mathbb{N}$ and $\delta>0$ such that $N/m>\alpha-\delta$ in the sector away from the sequence $\{a_n\}_{n\in\bbN}$.  
\end{assumption}
\begin{remark}\label{rem3.2}
$(i)$ Throughout this work, we choose $\ln \, F(z)$ to define a holomorphic function in $\mathcal I_\varepsilon^c$. As $F(z)$ does not vanish in that sector, this is always possible. In particular, $\ln \, F(0)$ in Theorem \ref{t3.5}  will denote the limit of $\ln \, F(z)$ as $z \to 0$ inside $\mathcal I_\varepsilon^c$.  
\\[1mm]
$(ii)$ The examples considered in Section \ref{s5} satisfy such asymptotics. Furthermore, under certain assumptions, characteristic functions associated with Sturm--Liouville operators admit such a representation with $\alpha=1/2$, $m=2$, and $M=1$ for every $N\in\mathbb{N}$ (see for instance \cite[Sec.~3]{FGKS21} and \cite[Sec.~4]{FPS25}).\\[1mm]
$(iii)$ More generally, one can replace the powers and logarithms in the asymptotic expansion \eqref{2.42} with a more general asymptotic sequence of functions endowed with an integrability assumption (see, e.g., \cite[Assumption 4.2]{FPS25}). The process described below can then be followed to investigate the structure of the analytic continuation of $\zeta_S(s)$. We refer to \cite[Sec.~5]{FPS25}, \cite{FS25}, and \cite{Ki06} for explicit examples of such problems, including those with generalized exponential integrals appearing in the asymptotic expansion of the characteristic function (leading to logarithmic branch points). In order to give explicit formulas to illustrate the method below, we choose to make the simpler assumption \eqref{2.42} in this work.
\end{remark}

Our next result reveals the pole structure of $\zeta_S(s)$ and its proof illustrates explicitly the role played by the asymptotic expansion \eqref{2.42} in the process of analytic continuation briefly described above.  
\begin{theorem}\label{t2.11}
Let Assumption \ref{assump2.9} hold and let $\bar{k}=\max\{0,\dots,M\}$ such that  $d_{j,\bar{k}}\neq 0$. Then the following holds regarding the pole structure of $\zeta_S(s)$$:$ 
\begin{equation}
\alpha-(j/m) =\begin{array}{l}
\begin{cases}
\text{pole of order $\bar{k}+1$}, & \text{if } (\alpha-(j/m))\notin\mathbb{Z},\\
\text{pole of order $\bar{k}$}, & \text{if }  (\alpha-(j/m))\in\mathbb{Z}\backslash \{0\},
\end{cases}
\end{array}
\end{equation}
where a pole of order $0$ is a regular point. 

Moreover, $s=0$ is a regular point for $\zeta_S(s)$ whenever $M\in \{0,1\}$. In this case we have
\begin{align}\label{2.56}
\zeta_S(0)&=\begin{cases}
0, & \text{if }\alpha-(j/m)\neq0 \text{ for any } j\in\{0,1,\dots,N\},\\
d_{j',1},& \text{if }\alpha-(j'/m)=0 \text{ for some } j'\in\{0,1,\dots,N\},
\end{cases}
\end{align}
noting that for $M=0$, one must have $d_{j',1}=0$.
For $M\geq 2$, $\zeta_S(s)$ develops, in general, a pole of order $M-1$ at $s=0$.
\end{theorem}
\begin{proof}
Utilizing the substitution $z=t e^{i\Psi}$, then
adding and subtracting the finitely many terms from \eqref{2.42} into the first integral in \eqref{2.28a} yields a representation which is valid in the region $\Re(s)>\alpha-(N/m)-\delta$:
\begin{align}\label{2.43}
&\zeta_S(s)=e^{is(\pi-\Psi)}\frac{\sin(\pi s)}{\pi}\int_R^{\infty}dt\, t^{-s}\frac{d}{dt}\left[\ln\,F\left(te^{i\Psi}\right)-\sum_{j=0}^N  \sum_{k=0}^{M} d_{j,k} e^{i(\alpha-(j/m)) \Psi} t^{\alpha-(j/m)}\ln^{k}\left(te^{i\Psi}\right)\right] \notag\\
& +e^{is(\pi-\Psi)}\frac{\sin(\pi s)}{\pi}\sum_{j=0}^N  \sum_{k=0}^{M} d_{j,k} e^{i(\alpha-(j/m)) \Psi}\int_R^\infty dt\, t ^{-s+\alpha-(j/m)-1}\ln^{k-1}\left(te^{i\Psi}\right)\left[k+(\alpha-(j/m))\ln\left(te^{i\Psi}\right)\right]\notag \\
&\quad-\frac{R^{-s}}{2\pi i}\int_{\Psi-2\pi}^{\Psi}d\theta\,e^{-si\theta}\frac{d}{d\theta}\ln\,F\left(Re^{i\theta}\right).
\end{align}
In this expression, the first integral, which we will denote by $Z_N(s)$, is now holomorphic for $\Re(s)>\alpha-(N/m)-\delta$. The integral in the second line represents a meromorphic function of $s\in\C$ that contains the information about the poles of $\z_S(s)$. The integral in the third line, which we will denote by $Q(s)$, is an entire function of $s$.
The integral in the second line of \eqref{2.43} can be computed explicitly by using the formula valid for $\mu<0$ and $n\in\N_0$
\begin{align}\label{intbyparts}
   \int_{R}^{\infty}dt\,t^{\mu-1}\ln^{n}\left(te^{i\Psi}\right)=-\frac{R^{\mu}}{\mu}\sum_{j=0}^{n}\frac{(-1)^{j}n!}{\mu^{j}(n-j)!}\ln^{n-j}\left(Re^{i\Psi}\right),
\end{align}
which can be proved by integrating by parts $n$-times. More explicitly, we have
\begin{align}\label{2.50}
&e^{is(\pi-\Psi)}\frac{\sin(\pi s)}{\pi}\sum_{j=0}^N  \sum_{k=0}^{M} d_{j,k} e^{i(\alpha-(j/m)) \Psi} \int_R^\infty dt\, t ^{-s+\alpha-(j/m)-1}\ln^{k-1}\left(te^{i\Psi}\right)\left[k+(\alpha-(j/m))\ln\left(te^{i\Psi}\right)\right]\nonumber\\
&\quad=e^{is(\pi-\Psi)}\frac{\sin(\pi s)}{\pi}\sum_{j=0}^N  \sum_{k=0}^{M}e^{i(\alpha-(j/m)) \Psi} \dfrac{d_{j,k}R ^{-s+\alpha-(j/m)}\,k!}{(s-\alpha+(j/m))^{k+1}}\Bigg[(\alpha-(j/m))\\
&\qquad\quad+\Theta(k-1)\sum_{l=0}^{k-1}\frac{\ln^{k-1-l}\big(R e^{i\Psi}\big)}{(s-\alpha+(j/m))^{l-k}(k-1-l)!}\Bigg(1+\frac{\alpha-(j/m)}{k-l}\ln\big(R e^{i\Psi}\big)\Bigg)\Bigg]=:\mathcal{L}_{\textrm{asy},N}(s),\notag
\end{align}
where $\Theta(\dott)$ denotes the Heaviside step-function
\begin{align}\label{eq:Heaviside}
\Theta(x)=\begin{cases}
    1,\quad x\geq 0,\\
    0,\quad x<0, 
\end{cases}\quad x\in\R. 
\end{align}
We can now express the spectral $\z$-function as a sum
\begin{align}\label{2.47}
   \zeta_S(s)=Z_{N}(s)+ \mathcal{L}_{\textrm{asy},N}(s)+Q(s).
\end{align}
This form makes the poles explicit in the region $\Re(s)>\alpha-(N/m)-\delta$, noting that if $(\alpha-(j/m))\in\mathbb{Z}$, then the $\sin(\pi s)$ term will contribute a zero of multiplicity one at these values. For instance, if $(\alpha-(j/m))\notin\mathbb{Z},$ the pole at $s=\alpha-(j/m)$ is of order $\bar{k}+1$ where $\bar{k}=\max\{0,\dots,M\}$ such that  $d_{j,\bar{k}}\neq 0$ whereas the pole is of order $\bar{k}$ if $(\alpha-(j/m))\in\mathbb{Z}\backslash\{0\}$ (with the understanding that a pole of order $0$ is a regular point).

From the expression \eqref{2.50} it is not difficult to realize that $s=0$ is always a regular point for $\zeta_S(s)$ whenever $\alpha-(j/m)\neq 0$. In this case $\zeta_{S}(0)=0$ which is obtained by simply substituting $s=0$ in \eqref{2.47}.

If, on the other hand, there exists $j'\in\{0,1,\dots,N\}$ such that $\alpha-(j'/m)=0$, then $\z_S(s)$  is regular at $s=0$ for $M=0$ leading to $\zeta_{S}(0)=0$. The same can be said for $M=1$ since the only resulting pole $e^{is(\pi-\Psi)}R^{-s}d_{j',1}/s$, coming from \eqref{2.50}, is canceled by the $\sin(\pi s)$ prefactor giving $\z_S(0)=d_{j',1}$. When $M\geq 2$, instead, the $\z$-function develops, in general, a pole of order $M-1$ at $s=0$.
\end{proof}

The explicit analytic continuation performed in the proof of Theorem \ref{t2.11} allows us to evaluate the residue of $\z_{S}(s)$ at each of the poles. The results are summarized in the following theorem: 

\begin{theorem}\label{t2.12}
    Let the assumptions of Theorem \ref{t2.11} hold. The residue of $\z_{S}(s)$ at the pole $s=\alpha-(j/m)$ is given by $($where $\Theta(\dott)$ denotes the Heaviside step-function \eqref{eq:Heaviside}$)$
    \begin{align}\label{2.53a}
    \Res[\zeta_S(s);\ s=\alpha-(j/m)]&=\frac{\Theta(M)}{2\pi i}d_{j,0}(\alpha-(j/m))\left(e^{2\pi i(\alpha-(j/m))}-1\right)\nonumber\\
    &\quad\ +\frac{\Theta(M-1)}{2\pi i}d_{j,1}\left[(2\pi i(\alpha-(j/m))+1)e^{2\pi i(\alpha-(j/m))}-1\right]\nonumber\\
    &\quad\ +\frac{\Theta(M-2)}{2\pi i}\sum_{k=2}^{M}d_{j,k}(2\pi i)^{k-1}e^{2\pi i(\alpha-(j/m))}\left[2\pi i(\alpha-(j/m))+k\right].
\end{align}
In particular, for $M\geq 2$ the residue of $\z_{S}(s)$ at $s=\alpha-(j'/m)=0$ is given by
\begin{align}\label{2.54a}
   \Res[\zeta_S(s);&\ s=0]=\sum_{k=2}^{M}d_{j',k}(2\pi i)^{k-1}k.
\end{align}
\end{theorem}
\begin{proof}
    The residue of $\zeta_{S}(s)$ at each of its poles $s=\alpha-(j/m)$ is the coefficient of the term $(s-\alpha+(j/m))^{-1}$ of the Laurent expansion of $\mathcal{L}_{\textrm{asy},N}(s)$ in \eqref{2.50} around the point $\alpha-(j/m)$.
In order to find this coefficient we first need the expansion
\begin{align}\label{2.55}
    e^{is(\pi-\Psi)}\frac{\sin(\pi s)}{\pi}R^{-s}=\sum_{n=0}^{\infty}\mathcal{Y}_{n}(s-\alpha+(j/m))^{n},
\end{align}
with the coefficients $\mathcal{Y}_{n}$ given by 
\begin{align}\label{2.56a}
\mathcal{Y}_{n}&=\frac{1}{n!}\left[e^{is(\pi-\Psi)}\frac{\sin(\pi s)}{\pi}R^{-s}\right]^{(n)}_{s=\alpha-(j/m)}\nonumber\\
&=\frac{1}{2\pi i\,n!}\left[\exp\left(s(2\pi i)-\ln\big(R e^{i\Psi}\big)\right)-\exp\left(-s\,\ln\big(R e^{i\Psi}\big)\right)\right]^{(n)}_{s=\alpha-(j/m)}\nonumber\\
&=\frac{1}{2\pi i\,n!}R^{-\alpha+(j/m)}e^{-i(\alpha-(j/m))\Psi}\left[\left(2\pi i-\ln\big(R e^{i\Psi}\big)\right)^{n}e^{2\pi i(\alpha-(j/m))}-(-1)^{n}\ln^{n}\big(R e^{i\Psi}\big)\right].
\end{align}
By utilizing the series \eqref{2.55} and the expression \eqref{2.50} for $\mathcal{L}_{\textrm{asy},N}(s)$, one can find the residue of $\z_{S}(s)$ at $s=\alpha-(j/m)$ in terms of the coefficients $\mathcal{Y}_{n}$ in \eqref{2.56a} as follows  
    \begin{align}\label{2.57a}
\Res[\zeta_S(s);\ s=\alpha-(j/m)\neq 0]&=\sum_{k=0}^{M}e^{i(\alpha-(j/m))\Psi}R^{\alpha-(j/m)}d_{j,k}\Bigg[(\alpha-(j/m))\sum_{l=0}^{k}\mathcal{Y}_{l}\frac{k!\ln^{k-l}\big(R e^{i\Psi}\big)}{(k-l)!}\nonumber\\
&\qquad +\Theta(k-1)\sum_{l=0}^{k-1}\mathcal{Y}_{l}\frac{k!\ln^{k-1-l}\big(R e^{i\Psi}\big)}{(k-1-l)!}\Bigg].       
\end{align}
For the first sum over $l$ in \eqref{2.57a} we have
\begin{align}
    \sum_{l=0}^{k}\mathcal{Y}_{l}\frac{k!\ln^{k-l}\big(R e^{i\Psi}\big)}{(k-l)!}&=\frac{R^{-\alpha+(j/m)}e^{-i(\alpha-(j/m))\Psi}}{2\pi i}\Bigg[e^{2\pi i(\alpha-(j/m))}\sum_{l=0}^{k}\binom{k}{l}\left(2\pi i-\ln\big(R e^{i\Psi}\big)\right)^{l}\ln^{k-l}\big(R e^{i\Psi}\big)\nonumber\\
    &\quad-\ln^{k}\big(R e^{i\Psi}\big)\sum_{l=0}^{k}\binom{k}{l}(-1)^{l}\Bigg].
\end{align}
Since 
\begin{align}
    \sum_{l=0}^{k}(-1)^{j}\binom{k}{l}=\begin{cases}
1, & \text{if }k=0,\\
0,& \text{if }k\geq 1,
\end{cases}\quad\textrm{and}\quad \sum_{l=0}^{k}(-1)^{j}\binom{k}{l}\left(2\pi i-\ln\big(R e^{i\Psi}\big)\right)^{l}\ln^{k-l}\big(R e^{i\Psi}\big)=(2\pi i)^{k},
\end{align}
we obtain
\begin{align}\label{2.60}
\sum_{l=0}^{k}\mathcal{Y}_{l}\frac{k!\ln^{k-l}\big(R e^{i\Psi}\big)}{(k-l)!}= \frac{R^{-\alpha+(j/m)}e^{-i(\alpha-(j/m))\Psi}}{2\pi i}\times\begin{cases}
e^{2\pi i(\alpha-(j/m))}-1, & \text{if }k=0,\\
e^{2\pi i(\alpha-(j/m))}(2\pi i)^{k},& \text{if }k\geq 1.
\end{cases}
\end{align}
A similar calculation shows that for the second sum over $l$ in \eqref{2.57a} one has
\begin{align}\label{2.61}
 \sum_{l=0}^{k-1}\mathcal{Y}_{l}\frac{k!\ln^{k-1-l}\big(R e^{i\Psi}\big)}{(k-1-l)!} =\frac{R^{-\alpha+(j/m)}e^{-i(\alpha-(j/m))\Psi}}{2\pi i}\times\begin{cases}
e^{2\pi i(\alpha-(j/m))}-1, & \text{if }k=1,\\
k\,e^{2\pi i(\alpha-(j/m))}(2\pi i)^{k-1},& \text{if }k\geq 2.
\end{cases}  
\end{align}
By substituting the results \eqref{2.60} and \eqref{2.61} in \eqref{2.57} one obtains the claim \eqref{2.53a}. 
The residue at $s=0$ in \eqref{2.54a} can be readily obtained from \eqref{2.53a} by setting $M\geq 2$ and $\alpha-(j'/m)=0$.  
\end{proof}

Since we now focus on the computation of $\zeta'_{S}(0)$, we restrict our analysis to the case $M=1$. As shown in Theorem \ref{t2.11}, in this case the $\z$-function is regular at $s=0$ which implies that $\zeta'_{S}(0)$ is well defined. 
We remind the reader that, throughout, $\ln \, F(z)$ denotes a holomorphic function in $\mathcal I_\varepsilon^c$, and $\ln \, F(0)$ its analytic continuation as $z \to 0$ inside $\mathcal I_\varepsilon^c$, which might not coincide with the standard choice of the logarithm (see, e.g.,~Remark~\ref{Remark:3.8} $(ii)$).
\begin{theorem}\label{t3.5}
Let Assumption \ref{assump2.9} hold with $M=1$. Then 
$\zeta_{S}(s)$ is regular at $s=0$ and 
\begin{align}\label{2.51b}
    \zeta'_{S}(0)&=\begin{cases}
-\ln\, F(0), & \text{if }\alpha-(j/m)\neq0 \text{ for any } j\in\{0,1,\dots,N\},\\
i\pi d_{j',1}+d_{j',0}-\ln\, F(0),& \text{if }\alpha-(j'/m)=0 \text{ for some } j'\in\{0,1,\dots,N\}.
\end{cases}  
\end{align}
In particular, $\zeta'_S(0)=i\pi\zeta_S(0)+d_{j',0}-\ln\, F(0)$ where $d_{j',0}=0$ in the first case of \eqref{2.51b}.
\end{theorem}

\begin{proof}

According to Theorem \ref{t2.11}, $\z_S(s)$ is regular at $s=0$ for $M=1$, hence $\zeta'_S(0)$ is well defined. 
First, notice that 
\begin{align}\label{2.52b}
 Z'_N(0)&=  \int_R^{\infty}dt\,\frac{d}{dt}\Bigg[\ln\,F\left(te^{i\Psi}\right)-\sum_{j=0}^{N}\left(te^{i\Psi}\right)^{(\alpha-(j/m))}\left(d_{j,0}+d_{j,1}\ln\left(te^{i\Psi}\right)\right)\Bigg] \nonumber\\
 &=-\ln\,F\left(Re^{i\Psi}\right)+\sum_{j=0}^{N}\left( Re^{i\Psi}\right)^{(\alpha-(j/m))}\left(d_{j,0}+d_{j,1}\ln\left(Re^{i\Psi}\right)\right).
\end{align}
For the function $Q(s)$ one concludes, from the representation \eqref{2.30}, that
\begin{equation}\label{2.53b}
Q'(0)=\int_0^{R}dt\,\frac{d}{dt}\ln\,F\left(te^{i\Psi}\right)=\ln\, F(Re^{i\Psi})-\ln\, F(0).
\end{equation}
Finally, whenever $\a-(j'/m)=0$ for some $j'\in\{0,1,\dots,N\}$, one has
\begin{align}\label{2.54b}
  \mathcal{L}_{\textrm{asy},N}^{'}(0)=\left[i\pi-\ln\big(R e^{i\Psi}\big)\right]d_{j',1}- \sum_{\underset{j\neq j'}{j=0}}^{N}\left(Re^{i\Psi}\right)^{(\alpha-(j/m))}\left(d_{j,0}+d_{j,1}\ln\left(Re^{i\Psi}\right)\right),\end{align}
  while for  $\a-(j/m)\neq 0$ for any $j\in\{0,1,\dots,N\}$ we get
  \begin{align}\label{2.54c}
    \mathcal{L}_{\textrm{asy},N}^{'}(0) =- \sum_{j=0}^{N}\left(Re^{i\Psi}\right)^{(\alpha-(j/m))}\left(d_{j,0}+d_{j,1}\ln\left(Re^{i\Psi}\right)\right). 
  \end{align}
By adding the results \eqref{2.52b}, \eqref{2.53b} and either of \eqref{2.54b} or \eqref{2.54c} one obtains \eqref{2.51b}.
\end{proof}
\begin{remark}\label{Remark:3.8}
$(i)$ In physical applications, one is concerned with $\zeta_S'(0)$ where $S$ is the spectrum of a, generally, self-adjoint operator $T$ in order to define the $\zeta$-regularized functional determinant as $\det_\zeta(T):=\exp{(-\zeta'_S(0))}$. As such, the particular branch of $\ln \, F(0)$ appearing in Theorem \ref{t3.5} will not contribute to the $\zeta$-regularized functional determinant of $T$.\\[1mm]
$(ii)$ Suppose the sequence $\{a_n\}_{n\in\bbN}$ is bounded from below and consists of real numbers. Then the sum definition of the corresponding $\zeta$-function (with the branch cut chosen in the second quadrant following \cite{FGKS21,FPS25,KM03,KM04}) immediately yields that the imaginary part of $\zeta'_S(0)$ is simply $i\pi m$ where $m$ is the number of negative terms in the sequence $\{a_n\}_{n\in\bbN}$. This can be readily understood by noting that the finite sum of the negative numbers can be split from the infinite sum of the remaining positive terms and that for each negative number in the sequence one has
\begin{equation}
\dfrac{d}{ds}a_n^{-s}\bigg|_{s=0}=-\ln(a_n)=i\pi-\ln(|a_n|).
\end{equation}
On the other hand, the imaginary part of $\zeta'_S(0)$ can be extracted from \eqref{2.51b}, yielding the interesting equality in this case:
\begin{equation}\label{negativewinding}
m=\Re(d_{j',1})+(1/\pi)\Im[d_{j',0}-\ln \, F(0)],
\end{equation}
where $d_{j',\ell},\ \ell=0,1,$  are appropriate constants from the large-$z$ asymptotic expansion of $\ln\, F(z)$ if $\a-(j'/m)=0$, and zero otherwise. This implies that \eqref{2.51b} can be rewritten as
\begin{align}\label{2.57}
  \zeta'_{S}(0)=i\pi m+\Re(d_{j',0})-\pi\,\Im(d_{j',1})-\ln\, |F(0)|=i\pi m+\Re(d_{j',0})-\pi\,\Im(\zeta_S(0))-\ln\, |F(0)|,
\end{align}
whenever the sequence  $\{a_n\}_{n\in\bbN}$ consists of real numbers and is bounded from below.
\end{remark}

The explicit analytic continuation performed in the proof of Theorem \ref{t2.11} can be used to evaluate the $\z$-function at integer points no greater than $\alpha$ as long as they are regular points of $\zeta_S(s)$ (note the point $s=0$ has already been covered in equation \eqref{2.56} of Theorem \ref{t2.11}). In fact we have the following:
\begin{theorem}\label{t3.7}
Let Assumption \ref{assump2.9} hold to all orders, that is, for all $M,N\in\mathbb{N}$. Furthermore, let $n\in\Z\backslash\{0\}$ with $n\leq\alpha$ be a regular point of the function $\z_S(s)$. Then
    \begin{align}\label{3.26}
     \zeta_S(n)=\begin{cases}
-nb_n, & \text{if }\alpha-(j/m)\neq n \text{ for any } j\in\{0,1,\dots,N\}\,\textrm{and}\; 1\leq n\leq\alpha ,\\
n(d_{m(\alpha-n),0}-b_n),& \text{if }\alpha-(\tilde{j}/m)=n \text{ for some } \tilde{j}\in\{0,1,\dots,N\}\,\textrm{and}\; 1\leq n\leq\alpha ,\\
0, & \text{if }\alpha-(j/m)\neq n \text{ for any } j\in\{0,1,\dots,N\}\,\textrm{and}\; n<0 ,\\
n\,d_{m(\alpha-n),0},& \text{if }\alpha-(\tilde{j}/m)=n \text{ for some } \tilde{j}\in\{0,1,\dots,N\}\,\textrm{and}\; n< 0,
\end{cases}     
    \end{align}
with $b_n$ given in \eqref{t2.28}. If Assumption \ref{assump2.9} is only valid for some $M,N\in\mathbb{N}$, then the above holds under appropriate restrictions on $n$.
\end{theorem}
\begin{proof}
    In order to analyze $\zeta_S(s)$ at $s=n\in\Z\backslash\{0\}$ with $n\leq\alpha$, we need to use the analytically continued expression \eqref{2.47}. Since the integral in the definition of $Z_N(s)$ in \eqref{2.43} is a holomorphic function of $s$ for $\Re(s)>\alpha-(N/m)-\delta$, it is immediate to conclude that for $s=n\in\Z$ and $n\leq\alpha$, $Z_N(n)=0$. The function $Q(s)$ in \eqref{2.43} is entire and by noting that (see \cite[Thm. 4.1]{FGKS21})
    \begin{align}
     z^{-n}\frac{d}{dz}\ln\,F(z) = \sum_{j=1}^{\infty}jb_jz^{j-n-1},  
    \end{align}
    with the coefficients $b_j$ given in \eqref{t2.28},
  and by using Cauchy's residue theorem, we obtain  
    \begin{align}
        Q(n)=\begin{cases}
-nb_n,& \text{if }\,  1\leq n\leq \alpha,\\
0 & \text{if }\, n\leq 0.
        \end{cases}
    \end{align}

    The remaining contribution to $\z_S(n)$ with $n\leq\alpha$ comes from the asymptotic terms $\mathcal{L}_{\textrm{asy},N}(s)$ in \eqref{2.50}. From the explicit expression \eqref{2.50} one sees that if $\alpha-(j/m)\notin\Z$ for any $j\in\{0,1,\dots,N\}$, the prefactor $\sin(\pi s)$ leads to $\mathcal{L}_{\textrm{asy},N}(n)=0$. 

    Let us assume, now, that there exists $\tilde{j}\in\{0,1,\dots,N\}$ such that $\alpha-(\tilde{j}/m)=n$. Since $n$ is a regular point of $\z_S(s)$, according to Theorem \ref{t2.11} the only non-vanishing term of the asymptotic expansion \eqref{2.42} is $d_{m(\alpha-n),0}$, noting that $\tilde{j}=m(\alpha-n)$. From the expression \eqref{2.50} all the terms in the sum vanish as $s\to n$ except for the $\tilde{j}$ term for which we find
    \begin{align}
    \mathcal{L}_{\textrm{asy},N}(n)=\lim_{s\to n}\left[e^{is(\pi-\Psi)}\frac{\sin(\pi s)}{\pi}e^{i n\Psi} \dfrac{n\,d_{m(\alpha-n),0}R ^{-s+n}}{s-n}\right]=n\,d_{m(\alpha-n),0}.    
    \end{align}
\end{proof}
\begin{remark}
    If $F(z)$ is given by the Hadamard factorization, then the integral $Q(s)$ in \eqref{2.43} vanishes identically (since $b_{1}=\dots=b_{\lfloor\alpha\rfloor}=0$) and the contribution to the value of the $\z$-function at $s=n\in\Z\backslash\{0\}$ with $n\leq\alpha$ comes entirely from the asymptotic part $\mathcal{L}_{\textrm{asy},N}(n)$ (c.f. Remark \ref{remposint} $(ii)$).
\end{remark}

\section{\texorpdfstring{$\zeta$}{zeta}-function after a linear transformation of the defining sequence}\label{s4}

We now study the changes that occur to the $\z$-function if the sequence $S$ is modified by a general linear transformation. Let $A,B\in\C$ with $A\neq 0$ and define the new sequence $S_{A,B}=\{\lambda_n\}_{n\in\N}$ where $\lambda_n=Aa_n+B$. We will assume that $A,B$ are chosen so that $|\lambda_n|>0$ for $n\in\N$. Since the general linear transformation represents a combination of a translation and a rotation in the complex plane, if the sequence $S$ satisfies the assumptions of Definition \ref{defzeta}, then $S_{A,B}$ satisfies those assumptions as well.
In fact, item $(i)$ of Definition \ref{defzeta} holds after a rearrangement of finitely many terms and item $(ii)$ is immediate since $a_n$ and $\lambda_n$ clearly have the same large-$n$ asymptotic behavior. Moreover, item $(iii)$ holds via the following result:

\begin{lemma}
Suppose the sequence $\{a_n\}_{n\in\mathbb{N}}$ satisfies Definition \ref{defzeta} $(iii)$ for some $\varepsilon>0$ and $\Psi\in[-\pi,\pi)$. Then the sequence $\{\lambda_n\}_{n\in\N}$ defined via $\lambda_n=Aa_n+B$ also satisfies Definition \ref{defzeta} $(iii)$ for some $\varepsilon'>0$ and $\Psi'\in[-\pi,\pi)$ such that $\mathcal{I}^c_{\varepsilon'}\subset A\mathcal{I}^c_{\varepsilon}+B$ for large enough $|z|$ where $A\mathcal{I}^c_{\varepsilon}+B=\{z\in\mathbb{C}:z=Aw+B\text{ for }w\in\mathcal{I}^c_\varepsilon\}$.
\end{lemma}
\begin{proof}
Notice that multiplication by $A$ represents a simple rotation by $\arg A$ and, therefore, Definition \ref{defzeta} $(iii)$ holds for $\{\lambda_n\}_{n\in\N}$ in the sector obtained by rotating $\mathcal{I}^c_\varepsilon$ by the same amount (here the superscript $c$ denotes the set complement). So it suffices to consider $A=1$. In addition, if the origin is still contained in the image of the sector $\mathcal{I}^c_\varepsilon$ under translation by $B$, the result is trivially true. So, we now assume $A=1$ and the origin is not contained in $\mathcal{I}^c_\varepsilon+B$.

In this setting, consider a ray emanating from the origin which is fully contained in $\mathcal{I}^c_\varepsilon+B$ for sufficiently large $|z|$. Since this ray will intersect $\mathcal{I}^c_\varepsilon+B$ at a finite distance, say $r>0$, the ray can further be chosen such that it does not intersect any of the terms of the sequence $\{\lambda_n\}_{n\in\N}$ as only finitely many can lie in the ball centered at the origin with radius $r$. Denote the argument defining such a ray by $\Psi'\in[-\pi,\pi)$. Now, consider the elements of the sequence $\{\lambda_n\}_{n\in\N}$, belonging to the ball of radius $r$, whose argument is closest to the argument $\Psi'$ of the ray $te^{i\Psi'}$. These elements (which are necessarily finite) are collectively denoted by $\lambda$ and their argument is $\nu$. Choosing $\nu>\varepsilon'>0$ defines a sector $\mathcal{I}^c_{\varepsilon'}$ about this line (i.e., the sector defined by $\Arg(z)\in[\Psi'-\varepsilon',\Psi'+\varepsilon']$) such that no members of $\{\lambda_n\}_{n\in\mathbb{N}}$ lie in this sector and $\mathcal{I}^c_{\varepsilon'}\subset\mathcal{I}^c_\varepsilon+B$ for large enough $|z|$. In particular, Definition \ref{defzeta} $(iii)$ holds for these choices.
\end{proof}

This implies that the $\z$-function of the sequence $S_{A,B}$ is well-defined and given by
\begin{align}
    \z_{S_{A,B}}(s)=\sum_{n=1}^{\infty}\lambda_n^{-s}=\sum_{n=1}^{\infty}(Aa_n+B)^{-s},\quad \Re(s)>\alpha.
\end{align}
By introducing the {\it shifted} $\z$-function of the sequence $S$ by an appropriate complex number $\mu$ 
\begin{align}
    \z_{S}(s;\mu)=\sum_{n=1}^{\infty}(a_n+\mu)^{-s},
\end{align}
we can write, since $A\neq 0$, 
\begin{align}\label{4.7}
   \z_{S_{A,B}}(s)=A^{-s}\z_{S}(s;B/A).
\end{align}
It is clear that the pole structure of $\z_{S_{A,B}}(s)$ is the same as the one for the shifted $\z$-function $\z_{S}(s;B/A)$, which, in turn, can be obtained from the analysis of $\z_{S}(s)$.

In fact, if the elements of the sequence $S$ are given as zeros of the characteristic function $F(z)$, then the terms $\{a_n+(B/A)\}_{n\in\N}$ are given as zeros of the characteristic function $F(z-(B/A))$. By computing the large-$z$ asymptotic expansion of $\ln\,F(z-(B/A))$ in terms of the one for $\ln\,F(z)$, we can obtain the analytic continuation of the shifted $\z$-function in terms of the one for $\z_S(s)$. The main result shows that the large-$z$ asymptotic expansion of $\ln\,F(z-(B/A))$ is of the same form as the one for $\ln\,F(z)$, a property which allows us to easily extend the theorems regarding the pole structure of $\z_S(s)$ to the shifted $\z$-function $\z_{S}(s;B/A)$. 

Throughout this section we choose the branch of $\ln\,F(z-(B/A))$ in $\mathcal{I}^c_{\varepsilon'}$ that agrees on $(\mathcal I^c_{\varepsilon} + B/A) \cap \mathcal I^c_{\varepsilon'}$ with the one selected for $\ln\,F(z)$.  We, then, have the following:
\begin{theorem}\label{t4.2}
    Let $F(z)$ be the characteristic function of the sequence $S$ in Definition \ref{defzeta}. Let $\ln\,F(z)$ satisfy the Assumption \ref{assump2.9}. Then the logarithm of the characteristic function $F(z-(B/A))$ associated with the shifted sequence $\{a_n+(B/A)\}_{n\in\N}$ has the large-$z$ asymptotic expansion 
    \begin{align}\label{t4.2a}
        \ln\,F(z-(B/A))=\sum_{j=0}^{N}\sum_{k=0}^{M}\Omega_{j,k}(A,B)z^{\alpha-(j/m)}\ln^{k}z+o\left(z^{\alpha-(N/m)-\delta}\right)\quad z\to\infty,\quad  z\in\mathcal{ I}_{\varepsilon'}^c.
    \end{align}
\end{theorem}
\begin{proof}
    Since $\ln\,F(z)$ satisfies the Assumption \ref{assump2.9}, the large-$z$ asymptotic expansion of $\ln\,F(z-(B/A))$ is of the form
    \begin{align}\label{4.9}
    \ln\,F(z-B/A)=\sum_{j=0}^N  \sum_{k=0}^{M} d_{j,k} (z-(B/A))^{\alpha-(j/m)}\ln^{k}(z-(B/A))+o\left(z^{\alpha-(N/m)-\delta}\right),\quad z\to\infty, \  z\in\mathcal{I}^c_{\varepsilon'}.    
    \end{align}
Now, as $z\to\infty$,
\begin{align}
 \ln^{k}(z-(B/A))&=\sum_{l=0}^{k}\binom{k}{l}\ln^{l}z\sum_{n=0}^{N}\frac{(k-l)!}{(k-l+n)!}s(k-l+n,k-l)(-B/A)^{k-l+n}z^{-n-k+l}+o\left(z^{-N-\delta}\right),
\end{align}
where we have utilized the generating function for the Stirling numbers of the first kind $s(n,k)$ \cite[Eq.~26.8.8]{DLMF}.
In addition, as $z\to\infty$, we have
\begin{align}
 (z-(B/A))^{\alpha-(j/m)}=z^{\alpha-(j/m)}\sum_{n=0}^{N}\binom{\alpha-(j/m)}{n}(-B/A)^{n}z^{-n}+o\left(z^{\alpha-(j/m)-N}\right).   
\end{align}
These last two expansions allow us to write
\begin{align}\label{4.12}
   (z-(B/A))^{\alpha-(j/m)}\ln^{k}(z-(B/A))&=z^{\alpha-(j/m)}\sum_{l=0}^{k}\binom{k}{l}\ln^{l}z\sum_{p=0}^{N}\mu_{p,l,k,j}z^{-p-k+l}+o(z^{\alpha-(j/m)-N-\delta}),
\end{align}
where, by the Cauchy product formula,
\begin{align}
\mu_{p,l,k,j}=(-B/A)^{k-l+p} \sum_{n=0}^{p}\frac{(k-l)!}{(k-l+n)!}\binom{\alpha-(j/m)}{p-n}s(k-l+n,k-l).
\end{align}

By substituting \eqref{4.12} in \eqref{4.9} and by using the expression 
\begin{align}
    \sum_{k=0}^{M}a_k\sum_{l=0}^{k}b_{l,k}\,\ln^{l}z=\sum_{k=0}^{M}\ln^{l}z\left(\sum_{n=0}^{M-k}a_{n+k}b_{k,n+k}\right),\quad a_k,b_{l,k}\in\C,
\end{align}
we obtain an asymptotic expansion ordered in increasing powers of $\ln\,z$
\begin{align}\label{4.15a}
\ln\,F(z-(B/A))=\sum_{k=0}^{M}\ln^{k}z\left[\sum_{n=0}^{M-k}\binom{n+k}{k}\sum_{j=0}^{N}\omega_{j,n,k}z^{\alpha-n-(j/m)}+o(z^{\alpha-M+k-(N/m)})\right],    
\end{align}
with 
\begin{align}\label{4.15b}
    \omega_{p,n,k}=\sum_{l=0}^{\lfloor p/m\rfloor}d_{p-lm,n+k}\mu_{l,k,n+k,p-lm}.
\end{align}
By reorganizing the sums in square parentheses in \eqref{4.15a} in decreasing powers of $z$ using the expression
\begin{align}
  \sum_{l=0}^{N}a_l\sum_{j\geq 0}b_{j,l}z^{-l-(j/m)}=\sum_{p\geq 0}z^{-(p/m)}\left(\sum_{l=0}^{\textrm{min}\{\lfloor p/m\rfloor,N\}}a_l b_{p-lm,l}\right),\quad a_l,b_{j,l}\in\C,
\end{align}
one finally obtains
\begin{align}
 \ln\,F(z-(B/A))=\sum_{j=0}^{N}\sum_{k=0}^{M}\Omega_{j,k}(A,B)z^{\alpha-(j/m)}\ln^{k}z+o\left(z^{\alpha-(N/m)-\delta}\right),  
\end{align}
where
\begin{align}\label{4.19}
 \Omega_{j,k}(A,B)=\sum_{l=0}^{\textrm{min}\{\lfloor j/m\rfloor,M-k\}}\binom{l+k}{k}\omega_{j-lm,l,k},  
\end{align}
completing the proof.
\end{proof}

Since the large-$z$ asymptotic expansion of $\ln\,F(z-(B/A))$ is of the same form as the one for $\ln\,F(z)$ in Assumption \ref{assump2.9}, all the results of Section \ref{continuation} hold true for the shifted $\z$-function $\z_{S}(s;B/A)$ with the replacement $d_{j,k}\to\Omega_{j,k}(A,B)$ and other simple modifications. For instance, by repeating the proof of Theorems \ref{t2.11} and \ref{t2.12}, with the obvious modifications, we obtain the following:
\begin{theorem}\label{t4.3}
    Let Theorem \ref{t4.2} hold and let $\bar{k}=\textrm{max}\{0,\ldots,M\}$ such that $\Omega_{j,\bar{k}}(A,B)\neq 0$. Then the following holds regarding the pole structure of $\z_{S_{A,B}}(s)$$:$
    \begin{equation}
\alpha-(j/m) =\begin{array}{l}
\begin{cases}
\text{pole of order $\bar{k}+1$}, & \text{if }(\alpha-(j/m))\notin\mathbb{Z},\\
\text{pole of order $\bar{k}$}, & \text{if }(\alpha-(j/m))\in\mathbb{Z}\backslash \{0\},
\end{cases}
\end{array}
\end{equation}
where a pole of order $0$ is a regular point. The residue of $\zeta_{S_{A,B}}(s)$ at the pole $\alpha-(j/m)$ is given by 
\begin{align}
    \Res[\z_{S_{A,B}}(s);s=\alpha-(j/m)]=A^{-\alpha+(j/m)}\Res[\z_{S}(s;B/A);s=\alpha-(j/m)].
\end{align}

Moreover, $s=0$ is a regular point for $\zeta_{S_{A,B}}(s)$ whenever $M=\{0,1\}$. In this case we have
\begin{align}\label{4.22}
\zeta_{S_{A,B}}(0)&=\begin{cases}
0, & \text{if }\alpha-(j/m)\neq0 \text{ for any } j\in\{0,1,\dots,N\},\\
\Omega_{j',1}(A,B),& \text{if }\alpha-(j'/m)=0 \text{ for some } j'\in\{0,1,\dots,N\},
\end{cases}
\end{align}
noting that for $M=0$, one must have $\Omega_{j',1}(A,B)=0$.
For $M\geq 2$, $\zeta_{S_{A,B}}(s)$ develops, in general, a pole of order $M-1$ at $s=0$ with residue  
\begin{align}
    \Res[\z_{S_{A,B}}(s);s=0]=\Res[\z_{S}(s;B/A);s=0].
\end{align}
Here, $\Res[\z_{S}(s;B/A);s=\alpha-(j/m)]$ and $\Res[\z_{S}(s;B/A);s=0]$ can be obtained from \eqref{2.53a} and \eqref{2.54a}, respectively, by replacing $d_{j,k}$ with $\Omega_{j,k}(A,B)$. 
\end{theorem}

In addition, we can easily compute the derivative of $\z_{S_{A,B}}(s)$ at $s=0$ (in the case this point is a regular point) by using a slight modification of Theorem \ref{t3.5}. 
Note that $\ln\,F(-B/A)$ is defined via analytic continuation of the branch of $\ln\,F(z)$ that is holomorphic in the sector $\mathcal{I}_{\varepsilon}^c$, extended along a path through $\mathcal{I}_{\varepsilon'}^c$.
More explicitly we have the following:
\begin{theorem}\label{t4.3a}
    Let Theorem \ref{t4.2} hold with $M=1$. Then $\z_{S_{A,B}}(s)$ is regular at $s=0$ and
    \begin{align}
    \zeta'_{S_{A,B}}(0)=-\ln\, F(-B/A),   
    \end{align}
    if $\alpha-(j/m)\neq0$ for any $j\in\{0,1,\dots,N\}$ and 
    \begin{align}\label{4.20a}
    \zeta'_{S_{A,B}}(0)=i\pi\,\Omega_{j',1}(A,B)+\Omega_{j',0}(A,B)-\Omega_{j',1}(A,B)\,\ln A-\ln\, F(-B/A),
    \end{align}
    if $\alpha-(j'/m)=0$ for some $j'\in\{0,1,\dots,N\}$.  
\end{theorem}
\begin{proof}
    From equation \eqref{4.7} we have
    \begin{align}\label{4.26}
     \z'_{S_{A,B}}(0)=-\ln\,A\,\z_{S}(0;B/A)+\z'_{S}(0;B/A).   
    \end{align}
By using an argument similar to the one employed to prove 
Theorem \ref{t3.5} we find
\begin{align}
\scalebox{0.9}{$\zeta'_{S}(0;B/A)=\begin{cases}
-\ln\, F(-B/A), & \text{if }\alpha-(j/m)\neq0 \text{ for any } j\in\{0,1,\dots,N\},\\
i\pi\,\Omega_{j',1}(A,B)+\Omega_{j',0}(A,B)-\ln\, F(-B/A),& \text{if }\alpha-(j'/m)=0 \text{ for some } j'\in\{0,1,\dots,N\}.
\end{cases}$}
\end{align}
By using this derivative and equation \eqref{4.22} (noting that $\z_{S_{A,B}}(0)=\z_{S}(0,B/A)$) in the relation \eqref{4.26} we obtain the claim.  
\end{proof}

Finally, it is not difficult to modify the proof of Theorem \ref{t3.7} and use \eqref{4.7} to verify the next result.
\begin{theorem}\label{t4.4}
Let Theorem \ref{t4.2} hold to all orders, that is, for all $M,N\in\mathbb{N}$. Furthermore, let $n\in\Z\backslash\{0\}$ with $n\leq\alpha$ be a regular point of the function $\z_{S_{A,B}}(s)$. Then
\begin{align}
\zeta_{S_{A,B}}\hspace{-1pt}(n)\hspace{-2pt}=\hspace{-2pt}\begin{cases}
-nA^{-n}B_n, & \hspace{-4pt} \text{if }\alpha-(j/m)\neq n \text{ for any } j\in\{0,1,\dots,N\}\,\textrm{and}\,1\leq n\leq\alpha ,\\
nA^{-n}(\Omega_{m(\alpha-n),0}(A,B)\hspace{-2pt}-\hspace{-2pt}B_n),& \hspace{-4pt} \text{if }\alpha-(\tilde{j}/m)=n \text{ for some } \tilde{j}\in\{0,1,\dots,N\}\,\textrm{and}\,1\leq n\leq\alpha ,\\
0, & \hspace{-4pt} \text{if }\alpha-(j/m)\neq n \text{ for any } j\in\{0,1,\dots,N\}\,\textrm{and}\, n\leq0 ,\\
nA^{-n}\Omega_{m(\alpha-n),0}(A,B),& \hspace{-4pt} \text{if }\alpha-(\tilde{j}/m)=n \text{ for some } \tilde{j}\in\{0,1,\dots,N\}\,\textrm{and}\,n\leq 0 ,
\end{cases}
\end{align}
where $B_{n}$ are the coefficients of the Taylor expansion of $(d/dz)\,\ln\,F(z-(B/A))$ about the point $z=0$.     
\end{theorem}

One of the consequences of Theorem \ref{t2.11} is that the order of any pole of $\z_S(s)$ cannot exceed $M+1$, where $M$ is the highest power of $\ln\,z$ appearing in the asymptotic expansion \eqref{2.42}. The results of Theorem \ref{t4.2} allow us to conclude that the same is true for the shifted $\z$-function; that is, $\z_{S_{A,B}}(s)$ cannot develop poles of order higher than those of $\z_{S}(s)$. The relation between the pole structure of $\z_{S}(s)$, intended as the description of the position and order of its poles, and that of $\z_{S_{A,B}}(s)$ is generally non-trivial. In fact, the pole structure of $\z_{S}(s)$ and that of $\z_{S_{A,B}}(s)$ depend explicitly on the coefficients $d_{j,k}$ of the expansion \eqref{2.42} and $\Omega_{j,k}$ of the expansion \eqref{t4.2a}, according to Theorems \ref{t2.11} and \ref{t4.3}, respectively. By using the expressions \eqref{4.15b} and \eqref{4.19} we can write the relation between these coefficients as
\begin{align}\label{4.29}
        \Omega_{j,p}(A,B)=\sum_{l=0}^{\textrm{min}\{\lfloor j/m\rfloor,M-p\}}\sum_{n=0}^{\lfloor (j/m)-l\rfloor}\binom{l+p}{p}d_{j-m(l+n),l+p}\mu_{n,p,l+p,j-m(l+n)}.
    \end{align}
This formula implies that there could be cases in which $\zeta_{S}(s)$ has a pole at a given point but the shifted $\z$-function does not. Indeed, it is clear that the terms $\Omega_{j,p}(A,B)$ are expressed as a linear combination of finitely many coefficients $d_{j,p}$ of the asymptotic expansion of $\ln\,F(z)$. It is, then, possible that a particular combination of coefficients $d_{j,p}$ could make one or more terms $\Omega_{j,p}(A,B)$ vanish. This means that a pole that is present for $\z_{S}(s)$ could be absent for $\z_{S_{A,B}}(s)$. For an explicit example of this phenomenon in the special case when $\z_{S}(s)$ has a simple pole at $s=1/2$ see \cite[Sec.~1]{FS25}.

Another interesting observation that can be made regarding \eqref{4.29} is that the coefficients $\Omega_{j,k}$ are written only in terms of the coefficients $d_{l,p}$ with $l\leq j$. This means that a pole of $\z_S(s)$ at the point $s=\alpha-l/m$ can contribute to the pole structure of $\z_{S_{A,B}}(s)$ only in the region $\Re(s)\leq \alpha-l/m$. This makes the first possible pole of $\z_{S_{A,B}}(s)$, namely the one at $s=\alpha$, special since $\z_{S}(s)$ has no poles for $\Re(s)>\alpha$. This simple observation implies that if the rightmost pole is present for $\z_{S}(s)$ then it will necessarily appear also for the shifted zeta function $\z_{S_{A,B}}(s)$.
\begin{corollary}\lb{cor4.7}
    If $s=\alpha$ is a pole of order $q$ for $\z_{S}(s)$, then it is also a pole of the same order for $\z_{S_{A,B}}(s)$. The residue at $s=\alpha$ is
    \begin{align}\label{c4.30}
        \Res[\z_{S_{A,B}}(s);s=\alpha]=A^{-\alpha}\Res[\z_{S}(s);s=\alpha].
    \end{align}
\end{corollary}
\begin{proof}
    The point $s=\alpha$ is a pole of order $q$ for $\z_{S}(s)$ only if $d_{0,q}\neq 0$ when $\alpha\in\Z\backslash\{0\}$ or $d_{0,q-1}\neq 0$ when $\alpha\notin\Z$. By using \eqref{4.29} we immediately see that $\Omega_{0,q}=d_{0,q}$ when $\alpha\in\Z\backslash\{0\}$ and  $\Omega_{0,q-1}=d_{0,q-1}$ when $\alpha\notin\Z$. Theorem \ref{t4.3} then ensures that $s=\alpha$ is a pole of $\z_{S_{A,B}}(s)$ of order $q$ with residue given by \eqref{c4.30} (after noting that  $\Res[\z_{S}(s;B/A);s=\alpha]$ coincides with $\Res[\z_{S}(s);s=\alpha]$ since $\Omega_{0,q}=d_{0,q}$ when $\alpha\in\Z\backslash\{0\}$ and  $\Omega_{0,q-1}=d_{0,q-1}$ when $\alpha\notin\Z$) .   
\end{proof}   

The results proven in this section can be illustrated explicitly by the following argument (which extends the discussion found in \cite[Sec.~1.2]{FS25}): By recalling that $\lambda_n=Aa_n+B$, with $A\neq 0$, and since $\{a_n\}_{n\in\N}$ is an increasing (in modulus) sequence of complex numbers (c.f. $(i)$ in Definition \ref{defzeta}), for any $B\in\C$ we can find an $m\in\N$ such that for all $n>m$, $|B|<|A||a_n|$. By using a negative binomial series
\begin{align}
    (x+c)^{-s}=\sum_{k=0}^{\infty}(-1)^{k}\binom{s+k-1}{k}x^k c^{-s-k},\quad |x|<|c|,
\end{align}
one can then write, for all $\lambda_n$ with $n>m$, 
\begin{align}
    \lambda_{n}^{-s}=A^{-s}\sum_{k=0}^{\infty}\frac{(-1)^{k}}{k!}\left(\frac{B}{A}\right)^k\frac{\Gamma(s+k)}{\Gamma(s)}a_{n}^{-s-k},\quad \Re(s)>0.
\end{align}
Fubini's theorem can then be used to obtain the following relation between $\z_{S_{A,B}}(s)$ and $\z_{S}(s)$, valid for any $m\in\bbN$ such that $|B|<|A||a_m|$, 
\begin{align}\lb{zetarelation}
  \z_{S_{A,B}}(s)=\sum_{n=1}^{m-1}\lambda_n^{-s}+A^{-s}\sum_{k=0}^{\infty}\frac{(-1)^{k}}{k!}\left(\frac{B}{A}\right)^k\frac{\Gamma(s+k)}{\Gamma(s)}\bigg[\z_{S}(s+k)-\sum_{n=1}^{m-1}a_n^{-s-k}\bigg], \quad \Re(s)>\alpha.
\end{align}
By analytic continuation, equation \eqref{zetarelation} can be utilized to study the singularity structure of $\z_{S_{A,B}}(s)$ if the poles of $\z_{S}(s)$ are already known.

For instance, from \eqref{zetarelation} it is easy to verify Corollary \ref{cor4.7}. In fact, assume that $\z_S(s)$ has a pole of order $q$ at $s=\alpha$. Since $\z_S(s)$ is holomorphic for $\Re(s)>\alpha$, the only pole (of the same order) appearing on the right-hand side of \eqref{zetarelation} comes from the $k=0$ term of the series. This implies that $\z_{S_{A,B}}(s)$ acquires a pole of the same order at $s=\alpha$ with residue given by \eqref{c4.30}.    

We would like to point out that while the relation \eqref{zetarelation} offers a very straightforward way of studying the connection between the poles and residues of the $\z$-function of $S$ and those of the shifted $\z$-function, it is not as useful for evaluating $\z_{S_{A,B}}(s)$ at specific points or its derivative at $s=0$ since, in general, we do not assume to know explicitly the elements of the sequence $S$.     

\begin{remark}
The goal of this section is to analyze $\z_{S_{A,B}}(s)$, not as a standalone function, but in terms of the $\z$-function associated with the original sequence $S$. For this reason, we did not present an analog of Theorem \ref{t4} for the $\z$-function of linearly transformed sequences. This can be specifically attributed to the fact that the proof of Theorem \ref{t4} relies on the Taylor series of $F(z)$ about the point $z=0$ and in order to prove a similar theorem corresponding to $\z_{S_{A,B}}(s)$ one would need to evaluate the Taylor series of $F(z)$ about the point $z=B/A$.  
While one can straightforwardly relate the coefficients of the large-$z$ asymptotic expansion of $\ln\,F(z-(A/B))$ to those of $\ln\,F(z)$, the same cannot be said for the corresponding Taylor series. To illustrate this point, simply consider the zeroth order coefficient of the Taylor series of $F(z)$ about $z=B/A$, namely $F(B/A)$. In order to write this coefficient in terms of the coefficients of the Taylor series of $F(z)$ about $z=0$ we would need to use the series $F(B/A)=\sum_{n=0}^{\infty}(F^{(n)}(0)/n!)(B/A)^n$.

Because of this problem, one should simply find the Taylor series of $F(z)$ about $z=B/A$ and then apply Theorem \ref{t4} directly. 
\end{remark}

\section{Examples of \texorpdfstring{$\zeta$}{zeta}-functions of sequences and the AAA algorithm}\label{s5}

We now utilize the results obtained in the previous sections to analyze the structure of the $\z$-function of certain sequences, illustrating the usefulness and power of the general method developed here. We begin with the classic Riemann $\z$-function, then turn to zeros of special functions that one often encounters in physical applications, ending with the AAA algorithm. The branch cut $R_{\Psi}$ will be chosen in the second quadrant throughout these examples unless otherwise stated.

\subsection{The Riemann \texorpdfstring{$\zeta$}{zeta}-function}

As a first example, we consider the $\z$-function associated with the sequence of positive integers $\{n\}_{n\in\N}$, that is the well-known Riemann $\z$-function, $\z_R(s)$. As characteristic function, one can take $F(z)=1/\Gamma(1-z),\ z\in\mathbb{C}$, where $\Gamma(\dott)$ is the standard Gamma function. Then $F(z)$ is entire with zeros at the positive integers and
\begin{align}\label{2.33}
F(z)&=1 - \gamma_E z + (1/12) \big(6 \gamma_E ^2 - \pi^2\big) z^2 + (1/12) \big(-2 \gamma_E ^3 + \gamma_E \pi^2 + 2 \psi''(1)\big)z^3\nonumber\\
&\quad + (1/1440)\big(60 \gamma_E ^4 - 60 \gamma_E ^2 \pi^2 + \pi^4 - 240 \gamma_E \psi''(1)\big) z^4 + O(z^5),
\end{align}
where $\gamma_E$ is the Euler--Mascheroni constant and $\psi(\dott)$ is the digamma function. Theorem \ref{t4} with $\alpha=1$ then yields\footnote{To apply Corollary \ref{c2.6} one would use $F(z)=e^{\gamma_E z}/\Gamma(1-z)$ since $\frac{1}{\Gamma(1-z)}=-\frac{1}{z\Gamma(-z)}=e^{-\gamma_E z}\prod_{k=1}^\infty \left(1-\frac{z}{k}\right)e^{z/k}$.} $($with $c_m$ defined via \eqref{2.33}$)$
\begin{align}
\begin{split}
\zeta_R(n)&=c_{1}c_{n-1}-nc_{n} -\sum_{\ell=2}^{n-1} c_{n-\ell}\zeta_R(\ell),\quad n\in\mathbb{N},\ n>1,\\
\zeta_R(2)&=c_1^2-2c_2=\gamma_E^2-(1/6)\big(6\gamma_E^2-\pi^2\big)=\pi^2/6,\quad \zeta_R(3)=c_1c_2-3c_3-c_1\zeta_R(2)=-\psi''(1)/2.
\end{split}
\end{align}
Using the known formula for positive even values of $\zeta_R(s)$ in terms of Bernoulli numbers, $B_n$, then yields
\begin{equation}
\zeta_R(2n)=c_{1}c_{2n-1}-2nc_{2n} -\sum_{\ell=2}^{2n-1} c_{2n-\ell}\zeta_R(\ell)=\dfrac{|B_{2n}|(2\pi)^{2n}}{2(2n)!},\quad n\in\mathbb{N}.
\end{equation}

Alternatively, one can start with $F(z)=z^{-1/2}\sin\big(\pi z^{1/2}\big)$ to study $\zeta_R(s)$ at even integers and obtain
\begin{equation}
\zeta_R(2n)=\dfrac{(-1)^{n+1} n\pi^{2n}}{(2 n+1)!} +\sum_{\ell=1}^{n-1}\dfrac{(-1)^{n-\ell+1} \pi^{2(n-\ell)}}{(2(n-\ell)+1)!}\zeta_R(2\ell),\quad n\in\mathbb{N},
\end{equation}
where we have used the series expansion
\begin{equation}
z^{-1/2}\sin\big(\pi z^{1/2}\big)=\sum_{k=0}^\infty \dfrac{(-1)^k \pi^{2k+1}}{(2 k+1)!} z^k,\quad z\in\mathbb{C}.
\end{equation}
We can utilize Theorem \ref{t3.5} to find $\zeta'_R(0)$ via the following large-$z$ asymptotic expansion valid for $\Arg(z)>0$:
\begin{align}\label{4.6}
-\ln \, \Gamma(1-z)=\left(z-\frac{1}{2}\right)\ln(z)-(i\pi +1)z-\frac{1}{2}\ln(2\pi)+\frac{i\pi}{2}+\sum_{k=1}^{N}\frac{B_{2k}}{2k(2k-1)z^{2k-1}}+O(z^{-2N-3}).    
\end{align}
Since $d_{1,1}$ is the coefficient of $\ln(t)$ and $d_{1,0}$ is the constant coefficient of the expansion, one finds that $d_{1,1}=-\frac{1}{2}$ and $d_{1,0}=-\frac{1}{2}\ln(2\pi)+\frac{i\pi}{2}$. By substituting these values in \eqref{2.56} and \eqref{2.57} one obtains the well-known results
\begin{align}
    \zeta_{R}(0)=-\frac{1}{2},\quad\textrm{and}\quad\zeta'_R(0)=-\frac{1}{2}\ln(2\pi).
\end{align}

Theorem \ref{t3.7} provides a formula to verify the values of the Riemann zeta function at negative integers. In fact, from \eqref{4.6} it is not difficult to see that the negative integers are regular points of $\z_R(s)$. We can, hence, apply the last of \eqref{3.26} to obtain, for $n\in\N$,
\begin{align}
    \z_{R}(-n)=-n\,d_{1+n,0}=-\frac{B_{n+1}}{n+1},
\end{align}
where the second equality holds since $d_{j,0}=B_j/j(j-1)$ which can be obtained by comparing the general asymptotic expansion \eqref{2.42} to the one in \eqref{4.6}.

\subsubsection{The Hurwitz \texorpdfstring{$\zeta$}{zeta}-function}

We now consider the case of the Hurwitz $\zeta$-function, $\zeta_H(s)$, to illustrate the results of Section \ref{s4}.
The Hurwitz $\z$-function is associated with the sequence of nonnegative integers translated by a complex number $a\in\bbC\backslash(-\bbN_0)$ that is
\begin{align}
    \z_{H}(s,a)=\sum_{n=0}^{\infty}(n+a)^{-s},\quad\Re(s)>1.
\end{align}
Its characteristic function is simply $F(z-a+1)$ where $F(z)=1/\Gamma(1-z)$, $z\in\bbC$ is the characteristic function of the Riemann $\z$-function considered earlier. Note that $a=1$ recovers, obviously, the Riemann $\zeta$-function example. 

In this case, $F(z-a+1)=1/\Gamma(a-z)$ is entire and its Taylor series about $z=0$ reads
\begin{align}
F(z-a+1)&=\frac{1}{\Gamma(a)}+\frac{\psi(a)}{\Gamma(a)} z +\frac{\psi^2(a)-\psi'(a)}{2\Gamma(a)} z^2 +\frac{\psi^3(a)-3\psi'(a)\psi(a)+\psi''(a)}{3!\Gamma(a)}z^3\nonumber\\
&\quad +\frac{\psi^4(a)-6\psi'(a)\psi^2(a) +4\psi''(a)\psi(a) +3[\psi'(a)]^2 -\psi^{(3)}(a)}{4!\Gamma(a)} z^4 + O(z^5).
\end{align}
One can then write down the positive integer values via Theorem \ref{t4}, which we omit here for brevity.

The theorems of Section \ref{s4} with $A=1$ and $B=a-1$ allow us to analyze the pole structure and the derivative at $s=0$ of $\z_{H}(s)$. First, since $s=1$ is a simple pole of the Riemann $\z$-function, by Corollary \ref{cor4.7}, $s=1$ is also a simple pole of $\z_{H}(s)$ with residue $1$.    
 
According to Theorem \ref{t4.2} and the expansion \eqref{4.6} we have that $\alpha=m=M=1$ and that the non-vanishing coefficients $\Omega_{j,k}$ are 
\begin{align}
    \Omega_{0,1}&=1,\quad\Omega_{1,1}=\frac{1}{2}-a,\quad \Omega_{0,0}=-(i\pi+1),\quad\Omega_{1,0}=-\frac{1}{2}\ln(2\pi)+i\pi\bigg(a-\frac{1}{2}\bigg),\nonumber\\
    \Omega_{j,0}&=\frac{(-1)^{j}(1-a)^{j-1}}{j(j-1)}\left(-\frac{j}{2}+1-a\right)+\sum_{n=1}^{\lfloor j/2\rfloor}\frac{B_{2n}}{2n(2n-1)}\binom{1-2n}{j-2n}(1-a)^{j-2n},\quad j\geq 2.
\end{align}
The coefficients $\Omega_{j,0}$ are indeed proportional to the Bernoulli polynomials. In fact, by using the expression for the Bernoulli polynomials in \cite[Eq.~24.2.5]{DLMF} with argument $1-a$ and the reflection formula \cite[Eq.~24.4.3]{DLMF} we find 
\begin{align}
    B_{j}(a)=(-1)^{j}\sum_{n=0}^{j}\binom{n}{k}B_{n}(1-a)^{j-n}.
\end{align}
Since $B_{0}=1$, $B_{1}=-1/2$ and $B_{2n+1}=0$ for $n\in\N$, we can separate even and odd terms in the previous expression to obtain 
\begin{align}\label{5.12}
   B_{j}(a)&=(-1)^{j}\left[(1-a)^{j}-\frac{j}{2}(1-a)^{j-1}\right]+(-1)^{j}\sum_{n=1}^{\lfloor j/2\rfloor}\binom{j}{2n}B_{2n}(1-a)^{j-2n}\nonumber\\
   &=(-1)^{j}(1-a)^{j}\left(-\frac{j}{2}+1-a\right)+(-1)^{j}\sum_{n=1}^{\lfloor j/2\rfloor}\binom{j}{2n}B_{2n}(1-a)^{j-2n}.
\end{align}
At this point we note that, since $1-2n$ is always a negative integer, we have
\begin{align}
    \binom{1-2n}{j-2n}=(-1)^{j}\frac{(j-2)!}{(j-2n)!(2n-2)!},
\end{align}
and therefore 
\begin{align}
    \frac{j(j-1)}{2n(2n-1)}\binom{1-2n}{j-2n}=\frac{(-1)^{j}j(j-1)(j-2)!}{2n(2n-1)(j-2n)!(2n-2)!}=\frac{(-1)^{j}j!}{(2n)!(j-2n)!}=(-1)^{j}\binom{j}{2n}.
\end{align}
By substituting this result into \eqref{5.12} and dividing the ensuing equation by $j(j-1)$ we arrive at the expression
\begin{align}
    \Omega_{j,0}=\frac{B_{j}(a)}{j(j-1)},\quad\textrm{for}\quad j\geq 2.
\end{align}

According to Theorem \ref{t4.3}, the Hurwitz $\z$-function has no poles other than the simple pole at $s=1$ mentioned earlier. The point $s=0$ is a regular point and by \eqref{4.22} we have for $a\in\bbC\backslash(-\bbN_0)$,
\begin{align}
    \z_{H}(0,a)=\Omega_{1,1}=\frac{1}{2}-a.
\end{align}

We can now exploit Theorem \ref{t4.3a} to compute the derivative of $\z_{H}(s,a)$ at $s=0$. In particular, \eqref{4.20a} leads to, with $a\in\bbC\backslash(-\bbN_0)$, 
\begin{align}\label{5.18}
    \z'_{H}(0,a)=-\frac{1}{2}\ln(2\pi)-\ln(1/\Gamma(a)),
\end{align}
where $\ln(1/\Gamma(a))$ is understood as the analytic continuation of $\ln(1/\Gamma(a-z))$ as $z \to 0$ inside a sector which includes the branch cut, $R_\Psi$, such that $\ln(1/\Gamma(a-z))$ is holomorphic (see Def. \ref{defzeta} and Rem. \ref{rem3.2} $(i)$).

Equation \eqref{5.18} can be made more explicit if we restrict our argument to $a\in\R\backslash(-\N_0)$ by following Remark \ref{Remark:3.8} $(ii)$ since in this case we can count the number $m$ of negative terms of the sequence. 
In particular, when $a<0$ one has $m=-\lfloor a\rfloor$. 
We can therefore conclude by \eqref{2.57} that 
\begin{equation}\label{5.21}
\zeta'_H(0,a)=\begin{cases}
-\frac{1}{2}\ln(2\pi)+\ln\, \Gamma(a),& \text{if $a>0$},\\
-i\pi \lfloor a\rfloor-\frac{1}{2}\ln(2\pi)+\ln\, |\Gamma(a)|,& \text{if $a<0$},
\end{cases}\quad a\in\bbR\backslash(-\bbN_0).
\end{equation}
Notice that for $a>0$, we easily recover, from \eqref{5.21}, the well-known result \cite[Eq.~25.11.18]{DLMF}. Equation \eqref{5.21} then extends \cite[Eq.~25.11.18]{DLMF} in $a$, which can also be verified via other formulas such as $\zeta_H(s,a)=\zeta_H(s,a+1)+a^{-s}$ \cite[Eq.~25.11.3]{DLMF}.

Lastly, by using Theorem \ref{t4.4} we can also obtain the values at $-n$, extending the classic formula \cite[Eq.~25.11.14]{DLMF}, valid for $\Re(a)>0$, to 
\begin{equation}
\zeta_H(-n,a)=-n\,\Omega_{1+n,0}=-\frac{B_{n+1}(a)}{n+1},\quad a\in\bbC\backslash(-\bbN_0),\, n\in\bbN.
\end{equation}

\subsection{Zeros of the Airy function} 
\label{Sect:Airy_Zeros}
In this example we consider the $\z$-function associated with the sequence of the negative of the infinitely many zeros $\{i_{n}\}_{n\in\N}$ of the Airy function $\Ai(z)$. The sequence $\{-i_{n}\}_{n\in\N}$ is then real and positive, $-i_{n}\to\infty$ as $n\to\infty$, and asymptotically one has $-i_{n}\sim n^{2/3}$ \cite[Sec.~9.9(iv)]{DLMF}. We denote its $\z$-function by
\begin{align}
    \z_{\Ai}(s)=\sum_{n=1}^{\infty}(-i_n)^{-s}, \quad \Re(s)>3/2.
\end{align}
The characteristic function associated with the above $\z$-function can be chosen to be $F(z)=\Ai(-z)$, $z\in\C$ as its zeros coincide with the sequence $\{-i_{n}\}_{n\in\N}$. The Airy function, $\Ai(-z)$, naturally appears when studying modified Fredholm determinants of the one-dimensional Schr\"odinger operator with linear potential on the half-line \cite{GK19a,Me16} (see also Sec.~\ref{s6} above regarding modified determinants). In addition, such a $\zeta$-function defined via Airy function zeros has also previously been studied, for instance, in \cite{Cr96} and \cite{voros23}. 

We now utilize the results of Theorem \ref{t4} to evaluate $\z_{\Ai}(n)$ with $n\in\N$. By using the relation $\Ai(-z)=(z^{1/2}/3)[J_{1/3}((2/3)z^{3/2})+J_{-1/3}((2/3)z^{3/2})]$ \cite[Eq.~9.6.6]{DLMF} and the small-$z$ expansion of the Bessel function \cite[Eq.~10.2.2]{DLMF} we obtain
\begin{align}\label{5.17a}
 F(z)=\sum_{k=0}^{\infty}\frac{(-1)^{k}3^{-2k-2/3}}{k!\Gamma\left(k+2/3\right)}z^{3k}+\sum_{k=0}^{\infty}\frac{(-1)^{k}3^{-2k-4/3}}{k!\Gamma(k+4/3)}z^{3k+1},   
\end{align}
from which one can easily extract the coefficients $c_j$ in Theorem \ref{t4}. The first few of them are
\begin{align}
    c_0=\frac{1}{3^{2/3} \Gamma \left(2/3\right)},\;c_{1}=\frac{1}{3^{1/3} \Gamma \left(1/3\right)},\;c_2=0,\;c_{3}=-\frac{1}{6 \left(3^{2/3} \Gamma \left(2/3\right)\right)},\;c_4=-\frac{1}{12 \left(3^{1/3} \Gamma \left(1/3\right)\right)}.
\end{align}
Since, in this case, $\alpha=2/3$, we can use \eqref{t2.25} to find, for example, the sums
\begin{align}
\begin{split}
    \zeta_{\Ai}(2)&=-2\frac{c_2}{c_0}+\left(\frac{c_1}{c_0}\right)^{2}=3^{2/3}\frac{\Gamma^{2}(2/3)}{\Gamma^{2}(1/3)},\quad
    \zeta_{\Ai}(3)=-3\frac{c_3}{c_0}+3\frac{c_{1}c_2}{c^2_0}-\left(\frac{c_1}{c_0}\right)^{2}=-3\frac{\Gamma^{3}(2/3)}{\Gamma^{3}(1/3)}+\frac{1}{2},\label{5.20}\\
    \zeta_{\Ai}(4)&=-4\frac{c_4}{c_0}-4\frac{ c_2 c_1^2}{c_0^3}+4\frac{c_3 c_1}{c_0^2}+2\left(\frac{ c_2}{c_0}\right)^{2}+\left(\frac{c_1}{c_0}\right)^4=3^{4/3}\frac{\Gamma^{4}(2/3)}{\Gamma^{4}(1/3)}-\frac{\Gamma(2/3)}{3^{2/3}\Gamma(1/3)},\\
    \zeta_{\Ai}(5)&=-3^{5/3} \frac{\Gamma^{5} \left(2/3\right)}{\Gamma^{5} \left(1/3\right)}+\frac{5}{4}\frac{\Gamma^2 \left(2/3\right)}{3^{1/3} \Gamma^2 \left(1/3\right)}.
\end{split}
\end{align}
Note that one can also easily study the $\zeta$-function associated with zeros of $\Ai'(-z)$ by taking the derivative of the series \eqref{5.17a}. In particular, this results in the new coefficients $\tilde{c}_m=(m+1)c_{m+1}$, yielding, for instance,
\begin{align}\label{5.20a}
    \zeta_{\Ai'}(2)&=-2\frac{3c_3}{c_1}+\left(\frac{2c_2}{c_1}\right)^{2}=\frac{\Gamma(1/3)}{3^{1/3}\Gamma(2/3)},\quad
    \zeta_{\Ai'}(3)=-3\frac{4c_4}{c_1}+3\frac{6c_{2}c_3}{c^2_1}-\left(\frac{2c_2}{c_1}\right)^{2}=1.
\end{align}
Equations \eqref{5.20} and \eqref{5.20a} recover the results found in \cite{voros23} including the curious fact that $\zeta_{\Ai'}(3)=1$.

To perform the analytic continuation of $\zeta_{\Ai}(s)$ to the left of $\Re(s)=2/3$, and utilize the result of Theorems \ref{t2.11} and \ref{t3.5}, we need the large-$z$ asymptotic expansion of $\Ai(-z)$. To this end, it is convenient to express the Airy function in terms of Hankel functions as follows \cite[Eq.~9.6.6]{DLMF}
\begin{align}
 \Ai(-z)=\frac{1}{2}\sqrt{\frac{z}{3}}\left(e^{-i\pi/6}H_{1/3}^{(2)}\left(\frac{2}{3}z^{3/2}\right)+e^{i\pi/6}H_{1/3}^{(1)}\left(\frac{2}{3}z^{3/2}\right)\right).   
\end{align}
According to \cite[Eqs. 10.2.5 and 10.2.6]{DLMF} the Hankel function $H_{1/3}^{(2)}((2/3)z^{3/2})$ is dominant compared to $H_{1/3}^{(1)}((2/3)z^{3/2})$ for $z\to\infty$ when $\Im(z)>0$. This implies that for large values of $z$ with $\Im(z^{3/2})$ we have, once the substitution $z=te^{i\Psi}$ has been performed,
\begin{align}\label{4.15}
    \ln\,F\left(te^{i\Psi}\right)&=-i\frac{2}{3}\left(te^{i\Psi}\right)^{3/2}-\frac{1}{4}\ln\left(te^{i\Psi}\right)+\frac{i\pi}{4}-\ln(2\sqrt{\pi})\nonumber\\
    &\quad+\ln\Bigg[1+\sum_{k=1   }^{N}\left(\frac{-3i}{2}\right)^{k}\frac{\Gamma(k+(1/6))\Gamma(k+(5/6))}{2\pi(-2)^{k}k!}\left(te^{i\Psi}\right)^{-3k/2}\Bigg]+O(t^{-3(N+1)/2}).    
\end{align}
By further expanding the logarithmic term in \eqref{4.15}
we obtain
\begin{align}\label{4.17a}
  \ln\,F\left(te^{i\Psi}\right)&=-i\frac{2}{3}\left(te^{i\Psi}\right)^{3/2}-\frac{1}{4}\ln\left(te^{i\Psi}\right)+\frac{i\pi}{4}-\ln(2\sqrt{\pi})+\sum_{j=1}^{N}p_{j}\left(te^{i\Psi}\right)^{-3j/2}+O(t^{-3(N+1)/2}).  
\end{align}
where the coefficients $p_{j}$ 
can be found by employing the fact
\begin{align}\lb{4.16}
    \ln\left(1+\sum_{m=1}^\infty C_m y^m\right) = \sum_{m=1}^\infty D_m y^m, \quad 0\leq |y|\text{ sufficiently small},
\end{align}
where
\begin{equation}\lb{4.17}
D_1=C_1, \quad D_j=C_j-\sum_{\ell=1}^{j-1} (\ell /j) C_{j-\ell}D_{\ell},\quad j \in \bbN, \; j\geq 2.
\end{equation}
Few of the fist coefficients read
\begin{align}
    p_1=\frac{5i}{48},\quad p_2=-\frac{5}{64},\quad p_3=-\frac{1105 i}{9216},\quad p_4=\frac{565}{2048},\quad p_5=\frac{82825 i}{98304},\quad p_6=-\frac{19675}{6144},
\end{align}
with higher order ones obtainable with the help of a simple computer program. 

The asymptotic expansion \eqref{4.17a} is of the form specified in Assumption \ref{assump2.9} with $M=1$, $\alpha=3/2$, and $m=2$. This means that we can use Theorem \ref{t2.11} to find all the poles of $\z_{\Ai}(s)$. In more detail, from the asymptotic expansion \eqref{4.17a} we find that the only non-vanishing coefficients $d_{j,k}$ are
\begin{align}\label{5.29}
    d_{0,0}=-\frac{2i}{3},\quad d_{3,0}=\frac{i\pi}{4}-\ln(2\sqrt{\pi}),\quad d_{3n+3,0}=p_{n}\, \quad n\in\N ,\quad \textrm{and}\quad d_{3,1}=-\frac{1}{4}.
\end{align}    
This implies that the points $s=3/2$ and $s=-(3/2)-3n$, $n\in\N_0$ are simple poles with residue (see Theorem \ref{t2.12})
\begin{align}\label{5.30}
 \Res[\zeta_{\Ai}(s);s=3/2]=\frac{1}{\pi}, \quad  \Res[\zeta_{\Ai}(s); s=-3/2-3n]=\frac{p_{2n+1}}{i\pi}\left(\frac{3}{2}+3n\right),\quad n\in\N_0.
\end{align}
Furthermore, $s=0$ is a regular point of $\z_{\Ai}(s)$ and the second formula in \eqref{2.56} can be used to obtain 
\begin{align}
\zeta_{\Ai}(0)=d_{3,1}=-\frac{1}{4}.    
\end{align}
Theorem \ref{t3.5} can now be applied to find $\z_{\Ai}'(0)$. In fact, by using $d_{3,1}$ and $d_{3,0}$ from \eqref{5.29} in equation \eqref{2.57}, with $m=0$, yields
\begin{align}\label{5.32}
\z_{\Ai}'(0)=-\ln(2\sqrt{\pi})-\ln|\Ai(0)|=\ln\left(\frac{3^{2/3}\Gamma(2/3)}{2\sqrt{\pi}}\right).\end{align}

Since the only poles occur at $s=3/2$ and $s=-(3/2)-3n$, $n\in\N_0$, the integers $s=-n$ with $n\geq -1$ are regular points  
of $\z_{\Ai}(s)$. This means that we can use Theorem \ref{t3.7} to compute the values of $\z_{\Ai}(s)$ at those points (except $s=0$). Since for each $n\in\Z\backslash\{0\}$ such that $n\geq -1$ we can find a $j$ such that $3-j=-2n$, we have from \eqref{3.26}, that 
\begin{align}\label{5.33}
   \z_{\Ai}(1)=d_{1,0}-\frac{c_1}{c_0}=-\frac{c_1}{c_0}=-3^{1/3}\frac{\Gamma(2/3)}{\Gamma(1/3)},
\end{align}
so that we find the \textit{universal exact sum rule} for $n\in\bbN$ (with $c_n$ given in \eqref{5.17a})
\begin{equation}
    \zeta_{\Ai}(n) = (-1)^n n\bigg(\frac{1}{n!}-3^{(2-n)/3}\frac{\Gamma \left(1/3\right)^n c_n}{\Gamma \left(2/3\right)^{n-1}}\bigg)\zeta_{\Ai}^n(1) + \sum_{\underset{n>k \geq 2}{j_1 + \, \dots \, + j_k = n}} \frac{(-1)^kn}{k!j_1 \cdots j_k} \zeta_{\Ai}(j_1) \cdots \zeta_{\Ai}(j_k).
\end{equation}
Interestingly, since $c_n=0$ for $n\equiv2\pmod 3$, we have
\begin{equation}
    \zeta_{\Ai}(n) = \frac{(-1)^n}{(n-1)!}\zeta_{\Ai}^n(1) + \sum_{\underset{n>k \geq 2}{j_1 + \, \dots \, + j_k = n}} \frac{(-1)^kn}{k!j_1 \cdots j_k} \zeta_{\Ai}(j_1) \cdots \zeta_{\Ai}(j_k),\quad n\equiv2\pmod3.
\end{equation}

For the negative integers we have, instead,
\begin{align}
  \z_{\Ai}(-n)=-n\,d_{2n+3,0},\quad n\in\N.   
\end{align}
Now, the only relevant non-vanishing coefficients of the asymptotic expansion \eqref{4.17a} are $d_{3k+3,0}$, $k\in\N_0$. This implies that whenever $2n=3k$ with $(n,k)\in\N$ the value of $\z_{\Ai}$ is non-zero. This occurs for $n=3j$ with $j\in\N$ and it leads to the result
\begin{align}
   \z_{\Ai}(-3j)=-3j\,d_{6j+3,0}=-3j\,p_{2j},\quad j\in\N.  
\end{align}
For all other negative integer values $\z_{\Ai}$ vanishes identically, that is
\begin{align}\label{5.33a}
 \z_{\Ai}(-n)=-n\,d_{2n+3,0}=0, \quad n\in\N\backslash\{3j\}, \quad j\in\N,  
 \end{align}
which represent the trivial zeros of the $\z$-function associated with the sequence of the negative of the zeros of the Airy function.

\begin{remark}
$(i)$ We would like to point out that rather than using the recursion \eqref{4.16} to find the coefficients $b_{2n+1}$ in the residues \eqref{5.30}, one could instead use the closed-form asymptotics proven in \cite{KM09} to write the residue as an exact integral representation for any $n$.\\[1mm]
$(ii)$ The paper \cite{Cr96} analyzes some particular values of the spectral $\z$-function associated with the Schrodinger equation endowed with a potential of the form $V(x)=|x|^{\nu}$. By assuming that the ensuing spectral $\z$-function is analytic with respect to the parameter $\nu$ the author conjectured that the analytically continued value of the Airy zeta function at $s=1$ was $\z_{\Ai}(1)=-3^{-2/3}\Gamma(2/3)/\Gamma(4/3)$. According to our expression in \eqref{5.33} the analytically continued result does in fact agree with this value (after the substitution $\Gamma(4/3)=(1/3)\Gamma(1/3)$).

We would like to add that the Airy zeta function was also studied in the work \cite{voros,voros23}. There, the author provided explicit values for $\z'_{\Ai}(0)$ and $\z_{\Ai}(n)$ for $n=1,2,3$, all of which agree with our results here (and those of \cite[Eq.~(4.23)]{Cr96}).
Furthermore, it is mentioned in \cite{voros23} that the value $\z_{\Ai}(1)$ provided there is meant as a `regularized sum' (also called regularized value or finite part), though it agrees with the value obtained here after analytic continuation.

It is interesting to point out that if we naively used Theorem \ref{t4} to evaluate $\zeta_{\Ai'}(s)$ at $s=1$ (which is a point outside of the validity of the theorem) we would obtain the value $0$. This
agrees with the regularized sum value provided in \cite[Eq.~(19)]{voros23}, indicating that the large modulus asymptotics of the integral representation \eqref{2.28a} once again does not contribute to this value (c.f. Remark \ref{remposint} $(i)$).\\[1mm]
$(iii)$ 
Finally, note that by \eqref{5.17a} one has for $k\geq2$,
\begin{equation}
\frac{c_k}{c_0}=\begin{cases}
\frac{(-1)^{k/3}}{(k/3)!3^{2k/3}[(k/3)-1+(2/3)][(k/3)-2+(2/3)]\dots(2/3)},& k\equiv 0\pmod 3,\\
\frac{(-1)^{(k-1)/3-1}}{((k-1)/3)!3^{2(k-1)/3}[((k-1)/3)-1+(4/3)][((k-1)/3)-2+(4/3)]\dots(4/3)}X,& k\equiv 1\pmod 3,\\
0,& k\equiv 2\pmod 3,
\end{cases}
\end{equation}
where
\begin{equation}
X=\zeta_{\Ai}(1)=-\frac{c_1}{c_0}=-3^{1/3}\frac{\Gamma(2/3)}{\Gamma(1/3)}.
\end{equation}
Therefore equations \eqref{2.34aa} and \eqref{t2.28} in Theorem \ref{t4} show that $\zeta_{\Ai}(n)$ is a polynomial in the variable $X=-3^{1/3}\Gamma(2/3)/\Gamma(1/3)$ with rational coefficients for all $n\in\bbN$. This shows that the conjecture put forth in \cite[Sec. 4]{Cr96} regarding the polynomial form of $\zeta_{\Ai}(n)$ for $n\geq  2$ is indeed correct, and in fact extends to $n=1$ by \eqref{5.33}. 

We further point out that this circle of ideas is quite general. In particular, due to the recursion \eqref{t2.28}, if $c_1\neq0$ and $c_k/c_0$ can be written as a polynomial in $X=-c_1/c_0$ with rational coefficients for each $k\in\mathbb{N}$, then $\zeta_S(n)$ can be written as a polynomial in the variable $X=-c_1/c_0$ with rational coefficients for all $n>\a$. We have chosen the negative sign in defining $X$ so that it agrees with $\zeta_S(1)$ whenever it is given by \eqref{2.34aa}, such as when $\a<1$ or as in the Airy example.
\end{remark}

\subsection{Zeros of the parabolic cylinder function}\label{Sect:Zeros_Para}
For $a>-1/2$, the parabolic cylinder function $U(a,z)$, defined in Section 12.5(i) of \cite{DLMF}, has no real zeros but has infinitely many zeros $\{q_n\}_{n\in\N}$ consisting of the union of $\{j_n\}_{n\in\N}$ and $\{\bar{j}_n\}_{n\in\N}$ which are complex conjugate of each other. As $n\to\infty$, $\Arg(j_n)\to3\pi/4$ while, obviously, $\Arg(\bar{j}_n)\to-3\pi/4$ \cite[Sec.~12.11]{DLMF}. Asymptotically, the zeros behave as $|q_{n}|\sim \sqrt{n}$ \cite[Sec.~12.11(ii)]{DLMF}, and hence, we can consider the $\z$-function associated with the sequence of zeros $\{q_n\}_{n\in\N}$ 
\begin{align}
    \z_U(s)=\sum_{n=1}^{\infty}q_n^{-s},\quad \Re(s)>2,
\end{align}
which we will show can be  analytically continued to all of $\mathbb{C}$.

Since for $a>-1/2$, $U(a,z)$ has no real zeros, we will choose the branch cut to be the positive real axis, $R_0$. The characteristic function associated with $\z_U(s)$ is $F(z)=U(a,z)$, $z\in\C$. In order to evaluate $\z_U(n)$, $n\in\N$ with $n\geq 3$, we need, according to Theorem \ref{t4}, the small-$z$ expansion of the characteristic function. To this end, we consider the following integral representation valid for $z\in\C$ and $a>-1/2$ \cite[Eq.~12.5.1]{DLMF}
\begin{align}
 U(a,z)=\frac{e^{-z^{2}/4}}{\Gamma(a+(1/2))}\int_{0}^{\infty}dt\,t^{a-1/2}e^{-t^2/2}e^{-zt}.   
\end{align}
By using the Taylor expansion for $e^{-zt}$ we find
\begin{align}
  U(a,z)= \frac{e^{-z^{2}/4}}{\Gamma(a+(1/2))}\sum_{n=0}^{\infty}\frac{(-1)^{n}}{n!}2^{\frac{2n+2a-3}{4}}\Gamma\left(\frac{2n+2a+1}{4}\right)z^{n}.
\end{align}
By further expanding the prefactor $e^{-z^{2}/4}$ and by subsequently using the Cauchy product we obtain
\begin{align}
    F(z)=U(a,z)=\sum_{j=0}^{\infty}c_{j}z^{j},
\end{align}
where for $j\in\N_0$ 
\begin{align}\label{4.28}
    c_{2j}&=\frac{2^{(2a-3)/4}}{\Gamma(a+(1/2))}\sum_{l=0}^{j}\frac{(-1)^{j-l}2^{l}}{4^{j-l}(2l)!(j-l)!}\Gamma\left(l+\frac{2a+1}{4}\right),\nonumber\\
    c_{2j+1}&=-\frac{2^{(2a-3)/4}}{\Gamma(a+(1/2))}\sum_{l=0}^{j}\frac{(-1)^{j-l}2^{l+1/2}}{4^{j-l}(2l+1)!(j-l)!}\Gamma\left(l+\frac{2a+3}{4}\right).
\end{align}
We can now use Theorem \ref{t4} with $\alpha=2$ and the coefficients \eqref{4.28} to find, for instance,
\begin{align}
    \zeta_{U}(3)&=2 \sqrt{2} \frac{\Gamma^3 ((2a+3)/4)}{\Gamma^3((2a+1)/4))}-\sqrt{2} a\frac{\Gamma((2a+3)/4)}{\Gamma((2a+1)/4))},\nonumber\\
    \zeta_{U}(4)&=4\frac{ \Gamma^4((2a+3)/4)}{\Gamma^{4} ((2a+1)/4))}-\frac{8a}{3}\frac{\Gamma^{2} ((2a+3)/4)}{\Gamma^{2} ((2a+1)/4))}+\frac{1}{12} \left(4 a^2-1\right),\nonumber\\
    \zeta_{U}(5)&=4\sqrt{2}\frac{\Gamma^5 ((2a+3)/4)}{\Gamma^5((2a+1)/4))}-\frac{10\sqrt{2}a}{3}\frac{\Gamma^3 ((2a+3)/4)}{\Gamma^3((2a+1)/4))}+\frac{\sqrt{2}}{24}\left(16 a^2-1\right)\frac{\Gamma((2a+3)/4)}{\Gamma((2a+1)/4))}.
\end{align}

Next, we analyze the structure of the poles of $\z_U(s)$. In the process of analytic continuation detailed in the proof of Theorem \ref{t2.11}, we deform the contour of integration to one that surrounds the branch cut $R_{\Psi}$. In this case, the deformed contour hugs the positive real line (our choice of branch cut) and, hence, the appropriate large-$z$ asymptotic expansion of $U(a,z)$ to utilize for the process of analytic continuation must be valid in a region containing the positive real line. This observation leads us to consider the following large-$z$ expansion of $U(a,z)$, and hence $F(z)$, valid for $a\in\C$ and $|\Arg(z)|<3\pi/4$ \cite[Eq.~12.9.1]{DLMF}   
\begin{align}
    F(z)=U(a,z)=\frac{e^{-z^{2}/4}}{z^{a+(1/2)}}\left(\sum_{n=0}^{N}\frac{(-1)^{n}}{2^{n}n!}\frac{\Gamma(2n+a+(1/2))}{\Gamma(a+(1/2))}z^{-2n}+O\left(z^{-2N-2}\right)\right).
\end{align}
From here, the expansion of $\ln\,F(t)$ (recall that $\Psi=0$) can be easily obtained by utilizing the relations \eqref{4.16} and \eqref{4.17}, that is
\begin{align}\label{4.31}
\ln\,F(t)=-\frac{1}{4}t^{2}-\left(a+\frac{1}{2}\right)\ln\,t+\sum_{j=1}^{N}h_jt^{-2j}+O\left(t^{-2N-2}\right),   
\end{align}
where the first few lower order coefficients $h_j$ are
\begin{align}
    h_{1}&=-\frac{1}{8} (2 a+1) (2 a+3),\quad h_{2}=\frac{1}{8} (2 + a) (1 + 2 a) (3 + 2 a),\nonumber\\
    h_3&=-\frac{1}{96} (2 a+1) (2 a+3) \left(20 a^2+88 a+99\right),\nonumber\\
    h_4&=\frac{1}{64} (2 a+1) (2 a+3) \left(28 a^3+200 a^2+489 a+408\right),\nonumber\\
    h_5&=-\frac{1}{320} (2 a+1) (2 a+3) \left(336 a^4+3424 a^3+13480 a^2+24232 a+16713\right),\nonumber\\
    h_6&=\frac{1}{192} (2 a+1) (2 a+3) \left(528 a^5+7136 a^4+39848 a^3+114632 a^2+169245 a+102096\right).
\end{align}

Also in this example, the form of the asymptotic expansion \eqref{4.31} satisfies Assumption \ref{assump2.9} with $M=1$, $\alpha=2$, and $m=1$. This implies that we can use Theorem \ref{t2.11} to extract the pole structure of $\z_U(s)$. 
By comparing \eqref{4.31} with the general form \eqref{2.42} we find that the only non-vanishing coefficients $d_{j,k}$ are
\begin{align}
    d_{0,0}=-\frac{1}{4},\quad d_{2n+2,0}=h_n,\quad n\in\N,\quad \textrm{and}\quad d_{2,1}=-a-\frac{1}{2}.
\end{align}
According to Theorem \ref{t2.11} the points $2-j\in\Z\backslash\{0\}$ are all regular which implies that $\z_U(s)$ develops no poles in $\C$. In particular, since $M=1$, the $\z$-function is regular at $s=0$ having value (c.f. \eqref{2.56}) 
\begin{align}
\z_U(0)=-a-\frac{1}{2}.   
\end{align}

Lastly, we use Theorem \ref{t3.5} to compute 
\begin{align}
    \z_{U}'(0)=\left(-a-\frac{1}{2}\right)i\pi-\ln\,U(a,0)=\left(-a-\frac{1}{2}\right)i\pi-\ln\left(\frac{\sqrt{\pi }\, 2^{-(2a+1)/4}}{\Gamma ((2a+3)/4)}\right),
\end{align}
where we have used the fact that $U(a,0)=c_0$ in \eqref{4.28} and the logarithm is understood as its analytic continuation as $z \to 0$ inside a sector which includes the positive real line such that $\ln\, U(a,z)$ is holomorphic (see Def. \ref{defzeta} and Rem. \ref{rem3.2} $(i)$).

Theorem \ref{t3.7} allows us to compute the values of $\z_{U}(n)$, $n\in\Z\backslash\{0\}$ with $n\leq 2$. More explicitly, from \eqref{3.26} and the asymptotic expansion \eqref{4.31} we find
\begin{align}
\z_{U}(1)&=d_{1,0}-\frac{c_1}{c_0}=2^{1/2}\frac{\Gamma((2a+3)/4)}{\Gamma((2a+1)/4)},\label{5.56aa}\\
\z_U(2)&=2d_{0,0}-2\frac{c_2}{c_0}+\left(\frac{c_1}{c_0}\right)^{2}=-a-\frac{1}{2}+2\frac{\Gamma^2 ((2a+3)/4)}{\Gamma^2((2a+1)/4))}.\label{5.57aa}
\end{align}
For $k\in\N$ we have instead
\begin{equation}
    \z_{U}(-2k)=-2k \,h_{k},\quad \z_U(-2k+1)=0.
\end{equation}
In this case we find that all of the negative odd integers are trivial zeros of $\z_U(s)$.

As a final comment, we would like to point out that this example exhibits a similar behavior that occurred in the case of the Airy $\z$-function. In fact, by setting  $X=2^{1/2}\Gamma((2a+3)/4)/\Gamma((2a+1)/4),$ $a>-1/2$, $\zeta_U(n)$ can be expressed as a polynomial in $X$ for all $n\in\bbN$ but with coefficients that are not rational in general except for special values of $a$. However, we should mention that for $n=1$ the analytic continuation agrees with \eqref{2.34aa} whereas it does not for $n=2$ (cf. \eqref{5.56aa} and \eqref{5.57aa}, resp.).

\subsection{Zeros of the confluent hypergeometric function}\label{Sect:Zeros_Conf}
As a final example, we consider the $\z$-function associated with the zeros of the confluent hypergeometric function, which we will show also admits an analytic continuation to all of $\mathbb{C}$. For 
$a$, $b$, and $b-a\not\in\Z_{\leq0}$, the confluent hypergeometric function $M(a,b,z)$, defined in \cite[Eq.~13.2.2]{DLMF}, has infinitely many zeros in $\C$ with respect to the argument $z$. We denote the sequence of zeros by $\{r_{n}\}_{n\in\N}$ ordered by increasing modulus. According to \cite[Sec.~13.9.9]{DLMF}, for fixed $a,b\in\C$, the zeros of $M(a,b,z)$ have the asymptotic behavior $|r_n|\sim \,n$ as $n\to\infty$ and, therefore, we can consider the associated $\z$-function
\begin{align}
    \z_M(s)=\sum_{n=1}^{\infty}r^{-s}_n,\quad \Re(s)>1.
\end{align}
In this case, we choose the branch cut $R_{\Psi}$ that avoids all the zeros of $M(a,b,z)$ and such that $-\pi/2<\Psi<\pi/2$. The characteristic function of $\z_M(s)$ is naturally given by $F(z)=M(a,b,z)$ with $a,b,b-a\notin\Z_{\leq0}$. 

The definition of $M(a,b,z)$ in \cite[Eq.~13.2.2]{DLMF} immediately provides the small-$z$ expansion of $F(z)$. Explicitly, we have
\begin{align}\label{4.41}
    F(z)=\sum_{n=0}^{\infty}c_{n}z^{n},\quad\textrm{with}\quad c_{n}=\frac{\Gamma(a+n)\Gamma(b)}{n!\Gamma(a)\Gamma(b+n)}.
\end{align}
According to Theorem \ref{t4}, with $\alpha=1$, it is, then, not difficult to obtain
\begin{align}
    \z_{M}(2)&=\frac{a(a-b)}{b^{2}(b+1)},\quad
    \z_M(3)=\frac{a (a-b) (b-2a)}{b^3 (b+1) (b+2)},\nonumber\\
    \z_M(4)&=\frac{a (a-b) \left(a^2 (5 b+6)-a b (5 b+6)+b^2 (b+1)\right)}{b^4 (b+1)^2 (b+2) (b+3)},\nonumber\\
    \z_M(5)&=\frac{a (a-b) (b-2a) \left(a^2 (7 b+12)-a b (7 b+12)+b^2 (b+1)\right)}{b^5 (b+1)^2 (b+2) (b+3) (b+4)}.
\end{align}
We would like to point out that these values confirm the ones obtained by Buchholz in \cite{Buch69}. 

In order to analyze $\z_{M}(s)$ for $\Re(s)\leq 1$, we consider the large-$z$ asymptotic expansion of $M(a,b,z)$ valid in the region of interest, $|\Arg(z)|<\pi/2$. This expansion can be obtained from the one for ${\bf M}(a,b,z)$ provided in \cite[Eq.~13.7.2]{DLMF} by noting that $M(a,b,z)=\Gamma(b){\bf M}(a,b,z)$. By performing the substitution $z=te^{i\Psi}$ we get
\begin{align}
    \ln\,F\left(te^{i\Psi}\right)&=te^{i\Psi}+(a-b)\ln\left(te^{i\Psi}\right)+\ln\left(\Gamma(b)/\Gamma(a)\right)\nonumber\\
    &\quad+\ln\left[1+\sum_{n=1  }^{N}\frac{\Gamma(1-a+n)\Gamma(b-a+n)}{\Gamma(1-a)\Gamma(b-a)n!}\left(te^{i\Psi}\right)^{-n}\right]+O\left(t^{-N-1}\right).
\end{align}
By using \eqref{4.16} and \eqref{4.17} we find
\begin{align}\label{4.44}
  \ln\,F\left(te^{i\Psi}\right)= te^{i\Psi}+(a-b)\ln\left(te^{i\Psi}\right)+\ln\left(\Gamma(b)/\Gamma(a)\right)+\sum_{j=1}^{N}f_j\left(te^{i\Psi}\right)^{-j}+O\left(t^{-N-1}\right),
\end{align}
where the first few coefficients $f_j$ are
\begin{align}
    f_1&=(a-1) (a-b),\quad f_2=-\frac{1}{2} (a-1) (a-b) (2 a-b-2),\nonumber\\
    f_3&=\frac{1}{3} (a-1) (a-b) \left(5 a^2-a (5 b+11)+b (b+6)+6\right),\nonumber\\
    f_4&=-\frac{1}{4} (a-1) (a-b) \left(14 a^3-a^2 (21 b+50)+a ( 9 b^{2}+53b+60)-b (b^{2}+12b+34)-24\right).
\end{align}

Also in this last example, the large-$z$ asymptotic expansion \eqref{4.44} is of the same form as \eqref{2.42} which implies that the theorems of the previous section can be applied. By comparing \eqref{4.44} and \eqref{2.42} we find that, in this case, $\alpha=M=m=1$ and the only non-vanishing coefficients $d_{j,k}$ are
\begin{align}
    d_{0,0}=1,\quad d_{1,1}=a-b,\quad d_{1,0}=\ln\left(\Gamma(b)/\Gamma(a)\right),\quad d_{j+1,0}=f_j,\quad j\in\N.
\end{align}
According to Theorem \ref{t2.11} the points $1-j\in\Z\backslash\{0\}$ are all regular, and therefore $\z_M(s)$ has no poles in $\C$. Moreover, from \eqref{2.56} we find
\begin{align}
    \z_M(0)=a-b.
\end{align}
From Theorem \ref{t3.5} we can compute the derivative of the $\z$-function at $s=0$,
\begin{align}
    \z_M'(0)=(a-b)i\pi+\ln\left(\Gamma(b)/\Gamma(a)\right)+\ln\, M(a,b,0),
\end{align}
where $\ln\, M(a,b,0)$ is understood as its analytic continuation as $z \to 0$ inside a sector including the branch cut, $R_{\Psi}$, such that $\ln\, M(a,b,z)$ is holomorphic (which will be equal to $2\pi i k$ for some $k\in\bbZ$ by \eqref{4.41}).

Lastly, we can compute the value of $\z_M(s)$ at $n\in\Z\backslash\{0\}$ such that $n\leq 1$ via Theorem \ref{t3.7}:
\begin{align}
    \z_M(1)=d_{0,0}-\frac{c_1}{c_0}=1-\frac{a}{b},\quad \z_M(-j)=-jf_j,\quad j\in\N. 
\end{align}

\subsection{Numerical verification using the AAA algorithm}
The adaptive Antoulas--Anderson (AAA) algorithm for rational interpolation introduced in \cite{AAA} has quickly established itself as a stable and efficient way to numerically compute analytic continuations in various settings. In particular, this method computes, in a fraction of a second, the analytic continuation of the Riemann $\z$-function in the critical strip and even for $\Re (s)< 0$ close to the origin, see \cite{NST23}, \cite{Trefethen23}. In the following we will use the AAA algorithm to numerically verify the special values of $\zeta_{\Ai}$ and $\zeta_{\Ai}'$ established in Section \ref{Sect:Airy_Zeros}. We demonstrate very close agreement with the theoretically computed values. The numerical analytic continuation of the other $\z$-functions considered in Sections \ref{Sect:Zeros_Para}, \ref{Sect:Zeros_Conf} should also be feasible, though more effort is needed due to the presence of zeros in the complex plane requiring more sophisticated root-finding algorithms. 

The AAA algorithm is implemented via the command \texttt{aaa} as part of the Chebfun package. The default relative interpolation accuracy is set to $10^{-13}$, hence we compute $\zeta_{\Ai}$  numerically up to this error at the interpolation points.  Note that while the series converges for $\Re (s) > 3/2$, the convergence rate is faster for larger $\Re (s)$. Therefore, we choose for the interpolation points $100$ equidistant points in the interval $[2, 8]$, requiring just $10^5$ terms to be included. Of these, we compute $10^3$ numerically, and for the remaining zeros we use the asymptotic formula \cite[Eq.~9.9.18]{DLMF}, including the subleading term. We also estimate the tail of the infinite series using the Euler--Maclaurin formula including the second correction. This allows us to compute $\zeta_{\Ai}(2)$ up to an error of size $\approx 10^{-13}$ as can be verified using the exact value from \eqref{5.20}.

The computation of the rational approximation takes a few second on a standard laptop.
In particular, the values of $\zeta_{\Ai}(1)$ and $\zeta_{\Ai}(0) = -1/4$ are computed up to $7$, resp.~$5$ correct digits. Furthermore, using a simple difference quotient with stepsize $10^{-6}$, we can compute $\zeta'_{\Ai}(0)$ up to for $4$ correct digits. The rational approximation also has an isolated zero around $\approx -0.992$ which is very close to the theoretical value $s = -1$ and even develops a pole at around $\approx -1.42$ which is in the vicinity of the theoretical value $s = -1.5$.  

Finally, we point out that such an algorithm can be used to numerically approximate the value of the corresponding $\zeta$-function at $s=-1/2$, a value important in the study of the Casimir energy as mentioned in the introduction. Our approximation gives us the value $\zeta_{\Ai}(-1/2) \approx -0.1393
$.

\end{document}